\documentclass[11pt]{article}

\usepackage[utf8]{inputenc}
\usepackage[margin=1in]{geometry}
\usepackage{amssymb}
\usepackage{amsmath}
\usepackage{amsthm}
\usepackage{enumitem}
\usepackage{bm}
\usepackage{tikz}
\usetikzlibrary{shapes}
\usepackage{subcaption}
\usepackage{graphicx}
\usepackage{xcolor}
\usepackage{thmtools}
\usepackage{thm-restate}
\usepackage[numbers]{natbib}
\usepackage{hyperref}
\usepackage{float}

\usepackage[linesnumbered,ruled]{algorithm2e}

\newtheorem{theorem}{Theorem}[section]

\newtheorem{lemma}[theorem]{Lemma}
\newtheorem{claim}{Claim}[theorem]

\newtheorem{cor}[theorem]{Corollary}
\newtheorem{prop}[theorem]{Proposition}

\theoremstyle{definition}

\newenvironment{subproof}[1][Proof]{\begin{proof}[#1]}{\end{proof}}

\newcommand{\mbn}{\mathbb{N}}

\newcommand{\exc}{\textup{exc}}
\newcommand{\tsp}{\textup{tsp}}
\newcommand{\minexc}{\textup{exc}}
\newcommand{\bminexc}{\widehat{\textup{exc}}}
\newcommand{\del}{\delta}
\newcommand{\bdel}{\widehat{\delta}}
\newcommand{\Del}{\Delta}
\newcommand{\BDel}{\widehat{\Delta}}
\newcommand{\cl}[1]{\overline{#1}}
\newcommand{\ec}{\mathcal{E}}
\newcommand{\bec}{\widehat{\mathcal{E}}}
\newcommand{\scan}{\textup{\texttt{Scan}}}

\newcommand{\alg}{\textup{\texttt{EC}}}
\newcommand{\balg}{\widehat{\textup{\texttt{EC}}}}

\newcommand{\subroutine}{\textup{\texttt{Subroutine}}}
\newcommand{\Algo}{\textup{\texttt{Algo}}}
\newcommand{\falg}{f_\Algo}
\newcommand{\flag}{\text{flag}}
\newcommand{\true}{\text{true}}
\newcommand{\false}{\text{false}}

\tikzstyle{vertex}=[fill=black, draw=black, shape=circle, minimum size = 5pt,inner sep=0pt]

\tikzstyle{edge}=[-]
\tikzstyle{directed edge}=[<-]

\makeatletter
\def\@fnsymbol#1{\ensuremath{\ifcase#1\or *\or 1\or 2\or 3\else\@ctrerr\fi}}
\makeatother
\title{Approximating TSP walks in subcubic graphs\thanks{This version corrects an error in the published version, J.~Combin.\ Theory Ser.~B 158 (2023) 70--104, \url{https://doi.org/10.1016/j.jctb.2022.09.002}. In statement~\ref{main1} of Theorem~\ref{thm:main}, the list of rooted graphs satisfying $\del(G,e)=-1$ was incomplete.
We also fix various typos and fill in degenerate cases in the proofs and the algorithm that were not explicitly stated in the previous version.
The results claimed in the abstract are not affected.}}
\author{Michael C. Wigal\thanks{Email: \texttt{michael.wigal.math@gmail.com}}
, Youngho Yoo\thanks{Current address: Department of Mathematics, Florida State University, Tallahassee, FL, USA. \newline Email: \texttt{yyoo2@fsu.edu}}
, and Xingxing Yu\thanks{Email: \texttt{yu@gatech.edu}}
}
\date{$^{1,3}$School of Mathematics \\ Georgia Institute of Technology \\ Atlanta, GA, USA \\
 \bigskip
 $^2$Department of Mathematics \\ Texas A\&M University \\ College Station, TX, USA \\
 \bigskip
\today}

\begin{document}

	\maketitle

		\begin{abstract}
    We prove that every simple 2-connected subcubic graph on $n$ vertices with $n_2$ vertices of degree 2 has a TSP walk of length at most $\frac{5n+n_2}{4}-1$, confirming a conjecture of Dvo\v r\'ak, Kr{\'a}l', and Mohar.
    This bound is best possible; there are infinitely many subcubic and cubic graphs whose minimum TSP walks have lengths $\frac{5n+n_2}{4}-1$ and $\frac{5n}{4} - 2$ respectively.
    We characterize the extremal subcubic examples meeting this bound.
    We also give a quadratic-time combinatorial algorithm for finding such a TSP walk.
    In particular, we obtain a $\frac{5}{4}$-approximation algorithm for the graphic TSP on simple cubic graphs, improving on the previously best known approximation ratio of $\frac{9}{7}$.
\end{abstract}

\newpage

	\section{Introduction} \label{sec:intro}

	The famous {\it Travelling Salesperson Problem} (TSP) asks for a spanning cycle of minimum length in an edge-weighted complete graph.
	It is not possible to approximate the TSP within any constant factor of the optimum unless $P = NP$; otherwise, one could solve the Hamiltonian cycle problem, one of Karp's original NP-complete problems \cite{Ka72}. An important special case which admits a constant factor approximation is the \textit{metric TSP} in which the edge weights form a metric, a natural assumption for many applications. A further specialization of the metric TSP is the {\it graphic TSP} in which the edge weights form the distance function in some underlying connected graph $G$ on the same vertex set.
	This is equivalent to finding a spanning closed walk (a \emph{TSP walk}) in $G$ with the minimum number of edges.
	Let us denote this minimum length by $\tsp(G)$.

	The graphic TSP still contains the Hamiltonian cycle problem, and is thus NP-hard to solve exactly.
	On the other hand,
	Christofides \cite{C76} and independently Serdyukov \cite{serdyukov1978nekotorykh, van2020historical} gave a $\frac{3}{2}$-approximation for the metric TSP in 1976 and 1978 respectively. For many years, this had remained the best approximation ratio for most nontrivial special cases of the graphic TSP.
	An improvement to this ratio was made in 2005 by Gamarnik, Lewenstein, and Sviridenko \cite{GLS04} who gave a $(\frac{3}{2} - \frac{5}{389})$-approximation algorithm for the special case of the graphic TSP on 3-connected cubic graphs (a graph is {\it cubic} if all of its vertices have degree 3). Following this result, Gharan, Saberi, and Singh \cite{GSS11} gave a $(\frac{3}{2}-\epsilon)$-approximation algorithm for the general graphic TSP.
	Then M{\"o}mke and Svensson \cite{MS11} gave a novel approach for a 1.461-approximation algorithm for the graphic TSP, which was shown to be in fact a $\frac{13}{9}$-approximation by Mucha \cite{M14}. Later, Seb{\H o} and Vygen \cite{SV14} presented a new algorithm for an improved $\frac{7}{5}$-approximation for the graphic TSP.
	For the metric TSP, the $\frac{3}{2}$ ratio was only very recently improved by Karlin, Klein, and Gharan \cite{karlin2021slightly} to $(\frac{3}{2} - \varepsilon)$ for some constant $\varepsilon > 10^{-36}$.

	A further special case of the graphic TSP, namely on subcubic graphs,  has received significant attention (a graph is {\it subcubic} if all of its vertices have degree at most 3).
	Subcubic and cubic graphs are among the simplest classes of graphs which retain the inapproximability of the metric TSP; the general metric and graphic TSPs are NP-hard to approximate within a factor of less than $\frac{123}{122}$ and $\frac{535}{534}$ of the optimum respectively \cite{karpinski2015new, KS15}.
    Even when restricted to subcubic and cubic graphs, it remains NP-hard to approximate within a factor of less than $\frac{685}{684}$ and $\frac{1153}{1152}$ respectively \cite{KS15}.
	Furthermore, subcubic graphs are known to exhibit the worst-case behavior in the well-known ``$\frac43$-integrality gap conjecture'' from the 80's (see \cite{G95}), which asserts that the standard ``subtour elimination'' linear program relaxation for the metric TSP has an integrality gap of $\frac{4}{3}$.
	This $\frac43$-integrality gap can be asymptotically realized by a family of subcubic graphs (e.g. \cite{benoit2008finding}).

	Note that a polynomial-time constructive proof of the $\frac{4}{3}$-integrality gap would yield a $\frac{4}{3}$-approximation algorithm.
	Motivated by this, Aggarwal, Garg, and Gupta \cite{AGG11} gave a $\frac{4}{3}$-approximation for 3-connected cubic graphs. This approximation ratio was extended to 2-connected cubic graphs by Boyd et al. \cite{BSSS14}, and to 2-connected subcubic graphs by M{\"o}mke and Svensson \cite{MS11}.
	The $\frac{4}{3}$ ratio was then slightly improved for cubic graphs to $(\frac{4}{3} - \frac{1}{61236})$ by
	Correa, Larr{\'e}, and Soto \cite{CLS15} and to $(\frac{4}{3} - \frac{1}{8754})$ by van Zuylen \cite{Z16}, which was further improved to 1.3 by Candr{\'a}kov{\'a} and Lukot'ka \cite{CL18}, and later to $\frac{9}{7}$ by Dvo\v r\'ak, Kr{\'a}l', and Mohar \cite{DKM17}.

    Let $G$ be a simple 2-connected subcubic graph.
    We write $n(G)$ to denote the number of vertices in $G$, and $n_2(G)$ to denote the number of degree 2 vertices in $G$.
    Dvo\v r\'ak, Kr{\'a}l', and Mohar \cite{DKM17} showed that $G$ has a TSP walk of length at most $\frac{9n(G)+2n_2(G)}{7}  -1$.
    They also constructed infinitely many subcubic (respectively, cubic) graphs whose minimum TSP walks have lengths $\frac{5n(G)+n_2(G)}{4} -1$ (respectively, $\frac{5n(G)}{4}-2$), and conjectured that $\frac{5n(G)+n_2(G)}{4} -1$ is the right bound. In this paper, we prove this conjecture.
    \begin{restatable}{theorem}{mainthm}
    \label{thm:tspwalk}
        Let $G$ be a 2-connected simple subcubic graph. Then  $\tsp(G)\le \frac{5n(G)+n_2(G)}{4} -1$. Moreover, a TSP walk of length at most $\frac{5n(G)+n_2(G)}{4} -1$ can be found in $O(n(G)^2)$ time.
    \end{restatable}
    In particular, we obtain a $\frac{5}{4}$-approximation algorithm for the graphic TSP on simple cubic graphs.
    We remark that our algorithm is purely combinatorial and deterministic. We also characterize the extremal examples of Theorem \ref{thm:tspwalk}; that is, the 2-connected simple subcubic graphs $G$ such that $\tsp(G) = \frac{5n(G)+n_2(G)}{4}-1$ (see Theorem \ref{thm:extremalcharacterization}). As pointed out by Dvo\v r\'ak et al. \cite{DKM17}, Theorem \ref{thm:tspwalk} is false for non-simple graphs. This can be seen from the graph obtained from three internally disjoint paths between two vertices, each of length $2k+1$, by the addition of parallel edges so that it becomes cubic.

    As in \cite{DKM17}, rather than working with Eulerian multigraphs obtained from spanning connected subgraphs by adding multiple edges (as often done in the literature), we consider spanning subgraphs $F$ of $G$ in which every vertex has degree 0 or 2. That is, $F$ is a spanning subgraph consisting of vertex-disjoint cycles and isolated vertices.
    We call such a subgraph $F$ an {\it even cover} of $G$.
    Let $c(F)$ denote the number of cycles in $F$ and $i(F)$ denote the number of isolated vertices in $F$.
    Define the {\it excess} of $F$ to be $$\exc(F)=2c(F)+i(F).$$
    For a graph $G$, let $\ec(G)$ denote the set of even covers of $G$, and define the {\it excess} of $G$ as
    $$\exc(G)=\min_{F\in \ec(G)}\exc(F).$$

    For example, consider the graph $\Theta$ which consists of three internally disjoint paths between two vertices, each path with $k$ vertices of degree 2. It is easy to see that an even cover consisting of a cycle and $k$ isolated vertices obtains the minimum excess. Thus for $k\geq 1$,
    \begin{align*}
        \exc(\Theta) = 2+k  \le \frac{(3k+2)+3k}{4} + 1 = \frac{n(\Theta)+n_2(\Theta)}{4} + 1,
    \end{align*}
    with equality when $k = 1$ (in which case $\Theta \cong K_{2,3}$).

    It is observed in  \cite{DKM17} that if $G$ is a connected subcubic graph, then there is an exact relation between $\tsp(G)$ and $\exc(G)$:
    \begin{prop}[Dvo\v r\'ak et al. \cite{DKM17}] \label{prop:evencoverwalk}
        Let $G$ be a connected subcubic graph. Then
        \begin{equation}
    \tsp(G)=\exc(G)-2+n(G).
    \end{equation}
    Moreover, an even cover $F\in \ec(G)$ can be converted into a TSP walk in $G$ of length $\exc(F)-2+n(G)$ in linear time.
    \end{prop}
    Thus, to prove Theorem~\ref{thm:tspwalk}, it suffices to show that
    \begin{equation}
        \exc(G) \le \frac{n(G)+n_2(G)}{4} + 1, \label{eqn:ecexc}
    \end{equation}
    and that an even cover $F$ of $G$ satisfying this bound can be found in quadratic time.
    Indeed, we will see that (\ref{eqn:ecexc}) follows from a more technical result (Theorem \ref{thm:main}) that bounds $\exc(F)- \frac{n(G)+n_2(G)}{4}$ for certain sets of even covers $F$ of $G$. In Section 2, we develop our key definitions and state Theorem \ref{thm:main}. In Section \ref{sec:thetachains}, we provide some technical lemmas on the structure of the extremal graphs for Theorem \ref{thm:main}, which we call \emph{$\theta$-chains}. We complete the proof of Theorem~\ref{thm:main} in Section \ref{sec:mainproof}. In Section \ref{sec:tight}, we characterize extremal graphs for Theorem \ref{thm:tspwalk}. In Section \ref{sec:algo}, we outline a quadratic-time algorithm that finds an even cover $F$ in simple 2-connected subcubic graphs $G$ with $\exc(F) \leq \frac{n(G)+n_2(G)}{4}+1$.

    We end this section with some notation. For a positive integer $k$, let $[k] = \{1,\ldots,k\}$.
    If $G$ and $H$ are graphs, we write $G \cup H$ (respectively, $G \cap H$) to denote the graph with vertex set $V(G) \cup V(H)$ (respectively, $V(G) \cap V(H))$ and edge set $E(G) \cup E(H)$ (respectively, $E(G) \cap E(H))$.
    Let $G$ be a graph. If $S$ is a set of vertices or a set of edges, we let $G - S$ denote the subgraph of $G$ obtained by deleting elements of $S$ as well as edges incident with a vertex in $S$. When $S=\{s\}$ is a singleton, we simply write $G-s$.
    For a collection $S$ of 2-element subsets of $V(G)$, we write $G + S$ for the graph with vertex set $V(G)$ and edge set $E(G) \cup S$. However, for $x,y\in V(G)$ we use $G+xy$ to denote the graph obtained from $G$ by adding a (possibly parallel) edge between $x$ and $y$.
    {For a subgraph $H \subseteq G$ and a set $S \subseteq V(G)$, we let $H + S$ denote the subgraph of $G$ such that $V(H+S) = V(H) \cup S$ and $E(H+S) = E(H)$.} For $S\subseteq V(G)$, we use $N(S)$ to denote the neighborhood of $S$ in $G$. If $S=\{s\}$ is a singleton, we simply write $N(s)$. When $|N(S)| \in \{1,2\}$, {\it suppressing} $S$ means deleting $S$ and adding a (possibly loop or parallel) edge between the vertices of $N(S)$.
    When $S=\{s\}$ is a singleton, \textit{suppressing $s$} means suppressing $\{s\}$.

	\section{Subcubic chains} \label{sec:minexcchains}

	In order to help with induction, we consider even covers which contain or avoid a specified edge.
	Let $G$ be a graph and let $e\in E(G)$.
	We write $\ec(G,e)$ to denote the set of even covers of $G$ containing $e$, and $\bec(G,e)$ to denote the set of even covers of $G$ not containing $e$.
	Define
	\begin{align*}
	    \minexc(G,e) &:= \min_{\substack{F\in \ec(G,e)}} \exc(F)- 2 \\
	    \bminexc(G,e) &:= \min_{\substack{F\in \bec(G,e)}} \exc(F)
	\end{align*}
	Clearly, we have $\exc(G) = \min\{\minexc(G,e)+2, \bminexc(G,e)\}$ for any edge $e\in E(G)$.
	The ``$-2$'' in the definition of $\minexc(G,e)$ leads to a natural interpretation of the quantities $\del(G,e)$ and $\bdel(G,e)$ defined below, and also results in simpler calculations as it accounts for the fact that the cycle $C$ of $F$ containing $e$ will often only be used as a path $C-e$ as part of a larger cycle (see Propositions \ref{prop:subcubicchains} and \ref{prop:contractchain}).

	To prove \eqref{eqn:ecexc}, it will be convenient to define the following parameters for a graph $G$ and an edge $e\in E(G)$:
     \begin{align*}
		\del(G,e) &:= \minexc(G,e) -\frac{n(G)+n_2(G)}{4},\\
		\bdel(G,e) &:= \bminexc(G,e) -\frac{n(G)+n_2(G)}{4}.
      \end{align*}
 Note that if every vertex of $G$ has degree 2 or 3 (for instance, if $G$ is subcubic and 2-connected), then $\del(G,e)$ and $\bdel(G,e)$ are always half-integral since $n(G)+n_2(G) = (n(G)-n_2(G))+ 2n_2(G)$ where $(n(G)-n_2(G))$ is the number of vertices of odd degree in $G$, which is always even.

	A \emph{subcubic chain} $C$ is a simple connected subcubic graph, written as an alternating sequence $C=xe_0B_1e_1B_2\dots B_ke_ky$ for some nonnegative integer $k$, satisfying the following properties (see Figure \ref{fig:subcubic_chain}):

	\begin{itemize}
	    \item $\{e_0,\dots, e_k\}$ is the set of cut-edges of $C$,
	    \item $\{B_0,B_1,\dots,B_k,B_{k+1}\}$ is the set of connected components of $C-\{e_0,\dots,e_k\}$, where $V(B_0)=\{x\}$ and $V(B_{k+1})=\{y\}$,
	    \item $B_i$ is either a single vertex or 2-connected for all $i\in[k]$, and
	    \item each $e_i$ has one endpoint in $B_i$ and one endpoint in $B_{i+1}$ for all $i=0,\dots,k$.
	\end{itemize}
		\begin{figure}[H]
	    \centering
	    \includegraphics[width=\textwidth]{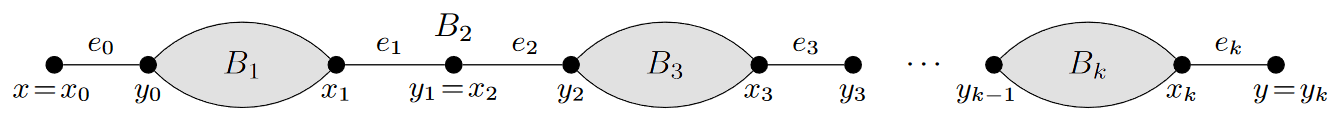}
	    \caption{A subcubic chain}
	    \label{fig:subcubic_chain}
	\end{figure}
	We say that $C$ has \emph{end points} $x,y$ and has \emph{end edges} $e_0$ and $e_k$.
	A subcubic chain is \emph{trivial} if $k=0$ (that is, $C$ is an edge $xy$), and \emph{nontrivial} otherwise.

	Let $C=xe_0B_1e_1B_2\dots B_ke_ky$ be a nontrivial subcubic chain.
	For $i\in [k]$, let $x_i$ denote the endpoint of $e_{i-1}$ in $B_i$ and let $y_i$ denote the endpoint of $e_i$ in $B_i$. (Note that $x_i\ne y_i$ when $n(B_i)\ne 1$, as $C$ is subcubic.)
	We define $\cl{B_i} = B_i + \cl{e_i}$ where $\cl{e_i} = x_iy_i$, and $\cl C = C-\{x,y\}+e_C$ where $e_C = x_1y_k$.
	We call each $(\cl{B_i},\cl{e_i})$ a \emph{chain-block} of $C$, and $\cl C$ the \emph{closure} of $C$. Note that the closure of a nontrivial subcubic chain $C$ is a subcubic graph with no cut-vertex such that $\cl C - e_C$ is simple.
	If $C$ is a trivial subcubic chain, we define $n(\cl C)=n_2(\cl C)=\minexc(\cl C,e_C)=\bminexc(\cl C,e_C)=\del(\cl C,e_C)=\bdel(\cl C,e_C)=0$.

	\begin{prop}	\label{prop:subcubicchains}
	Let $C=xe_0B_1e_1B_2\dots B_ke_ky$ be a subcubic chain, and let
	$\{(\cl{B_i},\cl{e_i}): i\in[k]\}$ denote the chain-blocks of $C$.
	Then
     \begin{itemize}
         \item $\minexc(\cl C,e_C) = \sum_{i=1}^k \minexc(\cl{B_i}, \cl{e_i})$,
         \item $\bminexc(\cl C,e_C) = \sum_{i=1}^k \bminexc(\cl{B_i},\cl{e_i})$,
         \item $\del(\cl C,e_C) = \sum_{i=1}^k \del(\cl{B_i}, \cl{e_i})$, and
         \item $\bdel(\cl C,e_C) = \sum_{i=1}^k \bdel(\cl{B_i},\cl{e_i})$.
     \end{itemize}
	\end{prop}

	\begin{proof}
	If $C$ is trivial then the proposition is true by definition (an empty sum is defined to be 0), and the case $k=1$ is immediate since $(\cl C,e_C)=(\cl B_1,\cl e_1)$.
	So we may assume $k\geq 2$.
	Note that a cycle in $\cl C$ contains $e_C$ if and only if it contains all of $e_1,\dots,e_{k-1}$.
	This gives a natural bijective correspondence between even covers $F\in \ec(\cl C,e_C)$ and tuples of even covers $(F_1,\dots,F_k)$ where $F_i \in \ec(\cl{B_i},\cl{e_i})$ for each $i\in [k]$.
	Indeed, this correspondence is obtained by ``splitting'' the cycle $D$ of $F$ containing $e_C$ into $k$ cycles,  $(D\cap \cl{B_i})+\cl{e_i}$ for $i\in[k]$.
	With this correspondence, we have $\exc(F) = 2 + \sum_{i=1}^k (\exc(F_i)-2)$.
	Hence,
	\begin{align*}
	    \minexc(\cl C,e_C) &= \min_{\substack{F\in \ec(\cl C,e_C)}} \exc(F) - 2 \\
	    &= \sum_{i=1}^k \min_{\substack{F_i \in \ec(\cl{B_i},\cl{e_i})}} (\exc(F_i)-2)\\
	    &= \sum_{i=1}^k \minexc(\cl{B_i}, \cl{e_i}).
	\end{align*}
	Since $n(\cl{C}) = \sum_{i=1}^k n(\cl{B_i})$ and $n_2(\cl{C}) = \sum_{i=1}^k n_2(\cl{B_i})$, this also implies $\del(\cl{C},e_C) = \sum_{i=1}^k \del(\cl{B_i},\cl{e_i})$.

	Similarly, there is a natural bijective correspondence between even covers $F\in \bec(\cl C,e_C)$ and tuples $(F_1,\dots,F_k)$ where $F_i \in \bec(\cl{B_i},\cl{e_i})$ for each $i\in[k]$. That is, $F_i$ is the restriction of $F$ on $B_i$ for all $i \in [k]$. Moreover, $\exc(F) = \sum_{i=1}^k \exc(F_i)$.  Hence,
	\begin{align*}
	    \bminexc(\cl C,e_C) &= \min_{\substack{F\in \bec(\cl C,e_C)}} \exc(F) \\
	    &= \sum_{i=1}^k \min_{\substack{F_i \in \bec(\cl{B_i},\cl{e_i})}} \exc(F_i)\\
	    &= \sum_{i=1}^k \bminexc(\cl{B_i}, \cl{e_i}).
	\end{align*}
	This similarly gives $\bdel(\cl C,e_C) = \sum_{i=1}^k \bdel(\cl{B_i},\cl{e_i})$.
	\end{proof}

	The parameters $\del(\cl C,e_C)$ and $\bdel(\cl C,e_C)$ can be interpreted as the ``difference'' in the $\del$ or $\bdel$ of the overall graph $G$ made by the presence of the subcubic chain $C$ compared to a trivial chain (a single edge).
	This is formalized in the next proposition.

	Let $G$ be a graph containing a nontrivial subcubic chain $C=xe_0B_1\dots B_ke_ky$ such that $C-\{x,y\}$ is a connected component of $G-\{e_0,e_k\}$.
	In this case, we say that $C$ is a \emph{subcubic chain of $G$}.
	If $C$ is a subcubic chain of $G$, we write $G/C$ to denote the graph obtained by suppressing $V(C)\setminus \{x,y\}$, and write $e_{G/C}$ to denote the resulting edge.
	We say that $G/C$ is obtained from $G$ by \emph{suppressing $C$}.
    A cycle in $G$ containing the edge $e_0$ (hence all of $\{e_0,\dots,e_k\}$) is said to be a \emph{cycle through} $C$, and an \textit{even cover through $C$}  is an even cover of $G$ containing a cycle through $C$.

	\begin{prop}\label{prop:contractchain}
	Let $C$ be a subcubic chain of a graph $G$, and let $e$ be a cut-edge of $C$.
	Then $\del(G,e)=\del(G/C,e_{G/C})+\del(\cl C,e_C)$ and $\bdel(G,e)=\bdel(G/C,e_{G/C})+\bdel(\cl C,e_C)$.
	\end{prop}
	\begin{proof}
	    Given an even cover $F\in \ec(G,e)$, $e$ is contained in some cycle $D$ in $F$. By splitting $D$ into two cycles $(D \cap G/C) + e_{G/C}$ and $(D \cap (C-\{x,y\})) + e_C$, we obtain from $F$ two even covers $F' \in \ec(G/C,e_{G/C})$ and $F_C \in \ec(\cl{C},e_C)$ satisfying $\exc(F) = \exc(F')+\exc(F_C)-2$. This bijective correspondence gives
	    \begin{align*}
	        \minexc(G,e)
	        &= \min_{\substack{F\in \ec(G,e)}} \exc(F) - 2 \\
	        &= \min_{\substack{F'\in \ec(G/C,e_{G/C})}} (\exc(F')-2)
	        + \min_{\substack{F_C\in \ec(\cl C,e_C)}} (\exc(F_C) - 2) \\
	        &= \minexc(G/C,e_{G/C})+\minexc(\cl C,e_C).
	    \end{align*}
	    Similarly, for any even cover $F\in \bec(G,e)$, its restriction on $G/C$ is in $\bec(G/C,e_{G/C})$ and its restriction on $\cl{C}$ is in $\bec(\cl{C},e_C)$; and we have $\bminexc(G,e) = \bminexc(G/C,e_{G/C})+\bminexc(\cl C,e_C)$.

	    Since $n(G) = n(G/C)+n(\cl C)$ and $n_2(G)=n_2(G/C)+n_2(\cl C)$, the proposition follows from the definitions of $\del$ and $\bdel$.
	\end{proof}

    We will show in Theorem \ref{thm:main} that $\del(G,e)+\bdel(G,e)\leq 0$ for every 2-connected subcubic graph $G$ and every edge $e \in E(G)$ for which $G-e$ is simple.
    If $\del(G,e)+\bdel(G,e)=0$, then we say that $(G,e)$ is \emph{tight}.
    A subcubic chain $C$ is {\it tight} if its closure $(\cl C, e_C)$ is tight.

    The next proposition states that a subcubic chain is tight if and only if all of its chain-blocks are tight.

	\begin{prop}\label{prop:tight}
	Let $C=xe_0B_1e_1B_2\dots B_ke_ky$ be a subcubic chain, and assume $\del(\cl{B_i},\cl{e_i})+\bdel(\cl{B_i},\cl{e_i})\leq 0$ for all $i$. Then
	$\del(\cl C,e_C)+\bdel(\cl C,e_C) \le 0$, with equality if and only if $\del(\cl{B_i},\cl{e_i})+\bdel(\cl{B_i},\cl{e_i}) = 0$ for all $i\in [k]$.
	\end{prop}
	\begin{proof}
	Since $\del(\cl{B_i},\cl{e_i})+\bdel(\cl{B_i},\cl{e_i})\leq 0$ for all $i$, we have by Proposition \ref{prop:subcubicchains},
	\begin{align*}
		\del(\cl C, e_C) &= \sum_{i=1}^k \del(\cl{B_i}, \cl{e_i})
		\leq \sum_{i=1}^k (-\bdel(\cl{B_i},\cl{e_i}))
		= -\bdel(\cl C,e_C).
	\end{align*}
	Hence, $\del(\cl C,e_C)+\bdel(\cl C,e_C) \leq 0$, with equality if and only if
	$\del(\cl{B_i},\cl{e_i})+\bdel(\cl{B_i},\cl{e_i}) = 0$ for all $i$.
	\end{proof}

    We say that a subcubic chain $C$ is \textit{minimal} if it is tight and   $\del(\cl C,e_C)=-\frac{1}{2}$, and that $C$ is \textit{near-minimal} if it is tight and $\del(\cl C,e_C)\in\{-\frac{1}{2},-1\}$.
	Two subcubic chains $C_1$ and $C_2$ are \textit{balanced} if $\del(\cl{C_1},e_{C_1})=\del(\cl{C_2},e_{C_2})$.

	A \emph{$\theta$-chain} is a graph $G$ that is the union of three internally disjoint subcubic chains $C_1,C_2,C_3$ with common endpoints.
	We call $C_1,C_2,C_3$ the \textit{chains of $G$}.
	Note that the choices of the three chains $C_1,C_2,C_3$ may not be unique (consider the graph obtained from two disjoint 4-cycles by adding two edges joining them so that the endpoints of the two edges are nonadjacent in each 4-cycle).
	A \textit{rooted $\theta$-chain} is a pair $(G,e)$ where $G$ is a graph and $e=uv\in E(G)$ such that $G-e$ is the union of two internally disjoint nontrivial subcubic chains $C_1,C_2$ with common endpoints $\{u,v\}$.
	We call $C_1,C_2$ the \textit{chains of $(G,e)$}. See Figure \ref{fig:theta_chain}.

	A (rooted) $\theta$-chain is \textit{balanced} if all pairs of its chains are balanced, \textit{tight} if the closures of its chains are all tight, and \textit{(near) minimal} if all of its chains are (near) minimal.
	Note that a minimal (rooted) $\theta$-chain is also balanced and tight by definition. See Figure \ref{fig:minimal_chain}.

	A \textit{rooted $K_4$-chain} is a pair $(G,e)$ where $e=uv\in E(G)$ and $G$ is the union of $e$, an edge $z_1z_2$ with $z_i\notin\{u,v\}$, and four internally disjoint (possibly trivial) subcubic chains $U_1,U_2,V_1,V_2$, where $U_i$ has endpoints $u,z_i$ and $V_i$ has endpoints $v,z_i$ for $i\in[2]$.
	We call $U_1,U_2,V_1,V_2$ the \textit{chains of $(G,e)$}.
	A rooted $K_4$-chain is \textit{tight} if its chains are all tight, and \textit{balanced} if $\del(\cl{U_1},e_{U_1})+\del(\cl{V_2},e_{V_2})=\del(\cl{U_2},e_{U_2})+\del(\cl{V_1},e_{V_1})$.
	Note that $(K_4,e)$ is a balanced tight rooted $K_4$-chain whose chains are all trivial.
\begin{figure}
\begin{minipage}{.5\textwidth}
  \centering	    \includegraphics[scale=0.35]{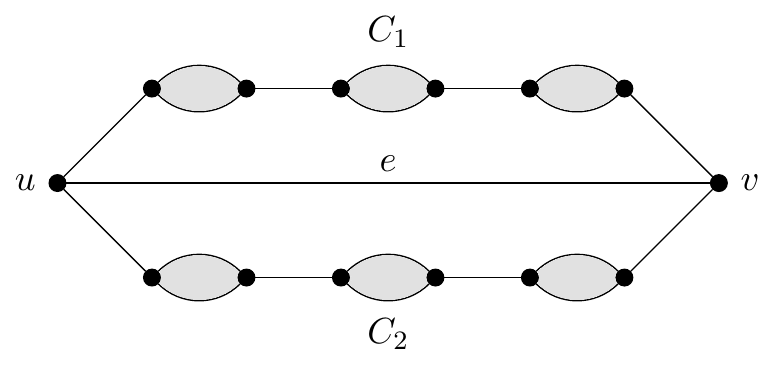}
  \caption{A rooted $\theta$-chain}
  \label{fig:theta_chain}
\end{minipage}%
\begin{minipage}{.5\textwidth}
    \centering

\begin{tikzpicture}
	\node[style=vertex] (u) at (0,0){};
	\node[style=vertex] (a) at (2,0){};
	\node[style=vertex] (v) at (4,0){};
	\node[style=vertex] (b1) at (1,1){};
	\node[style=vertex] (b2) at (2,1.4){};
	\node[style=vertex] (b3) at (2,0.6){};
	\node[style=vertex] (b4) at (3,1){};
	\node[style=vertex] (c1) at (1,-1){};
	\node[style=vertex] (c2) at (2,-0.6){};
	\node[style=vertex] (c3) at (3,-1){};
	\node[style=vertex] (c4) at (1.5,-1.5){};
	\node[style=vertex] (c5) at (2,-1.2){};
	\node[style=vertex] (c6) at (2.5,-1.5){};
	\node[style=vertex] (c7) at (2,-1.8){};

	\draw[style=edge] (u) to (a) to (v);
	\draw[style=edge] (u) to (b1) to (b2) to (b4) to (b3) to (b1);
	\draw[style=edge] (b4) to (v);
	\draw[style=edge] (u) to (c1) to (c2) to (c3) to (v);
	\draw[style=edge] (c1) to (c4) to (c5) to (c6) to (c7) to (c4);
	\draw[style=edge] (c6) to (c3);
\end{tikzpicture}
	\caption{A minimal $\theta$-chain}
	\label{fig:minimal_chain}
\end{minipage}
\end{figure}

	We can now state our main result, which immediately implies \eqref{eqn:ecexc}.
	For inductive purposes, we allow the graph $G$ to be a loop $e$ on a single vertex and we also allow one edge of $G-e$ to be parallel to $e$. In all cases however, $G-e$ is a simple subcubic graph. Note that a loopless subcubic graph with at least three vertices is 2-edge-connected if and only if it is 2-connected, and that the only other 2-edge-connected subcubic graphs $G$ with an edge $e$ such that $G-e$ is simple are the loop and the digon (two vertices joined by two parallel edges).

    \begin{restatable}{theorem}{maintheorem} \label{thm:main}
		Let $G$ be a 2-edge-connected subcubic graph and let $e=uv$ be an edge of $G$ such that $G-e$ is simple.
		Then the following statements hold:
		\begin{enumerate}[label=\textbf{\textup{(T\arabic*)}}]
			\item $\del(G,e) \leq -\frac{1}{2}$, with equality if and only if either $G$ is a loop or $(G,e)$ is a balanced tight rooted $\theta$-chain. \label{main1/2}
			\item If $G-e$ is 2-connected, then $\bdel(G,e) \leq \frac{3}{2}$, with equality if and only if $G-e$ is a minimal $\theta$-chain.  \label{main3/2}
			\item If $\del(G,e) = -1$, then either
			\begin{enumerate}
			    \item $(G,e)$ is a balanced tight rooted $K_4$-chain with chains $U_1,U_2,V_1,V_2$ such that $\del(\cl{U_1},e_{U_1})-\del(\cl{U_2},e_{U_2})\in\{-\frac{1}{2},0,\frac{1}{2}\}$, or
			    \item $e$ has a parallel edge, and either $n(G)=2$ or suppressing $\{u,v\}$ to an edge $e'$ results in a graph $G'$ such that either $G'$ is a loop or $(G',e')$ is a near-minimal rooted $\theta$-chain, or
			    \item there exists $e'\in E(G)$ such that $\{e,e'\}$ is a 2-edge-cut in $G$, and suppressing either subcubic chain $C$ of $G$ with end edges $e,e'$ yields either a loop or a balanced tight rooted $\theta$-chain $(G/C,e_{G/C})$, or
			    \item $(G,e)$ is a rooted $\theta$-chain such that $\min_{i\in[2]}\left(\del(\cl{C_i},e_{C_i})+\bdel(\cl{C_{3-i}},e_{C_{3-i}})\right) = -\frac{1}{2}$.
			\end{enumerate} \label{main1}
			\item $\del(G,e)+\bdel(G,e) \leq 0$. \label{maindelbdel}
		\end{enumerate}
    \end{restatable}
	One immediate consequence of Theorem \ref{thm:main} is that if $C$ is a subcubic chain, then $\del(\cl{C},e_C)\leq -\frac{1}{2}$ unless $C$ is trivial, in which case $\del(\cl{C},e_C)=0$ by definition. In particular, $\del(G,e)\leq -\frac{1}{2}$ for every nonempty 2-connected subcubic graph $G$ and $e\in E(G)$ such that $G-e$ is simple.
	Hence, if $C$ is a minimal subcubic chain, then by Proposition \ref{prop:subcubicchains}, it has exactly one chain-block $(\cl{B},\cl{e_B})$, and this chain-block satisfies $\del(\cl{B},\cl{e_B})=-\frac{1}{2}$.

\section{Properties of $\theta$-chains}
	\label{sec:thetachains}

	In this section, we derive useful properties of balanced, tight, or minimal $\theta$-chains assuming Theorem \ref{thm:main} for smaller graphs.
	We begin by proving statements \ref{main1/2} and \ref{main1} of Theorem \ref{thm:main}, assuming Theorem \ref{thm:main} for smaller graphs, for the special case where $(G,e)$ is a rooted $\theta$-chain (equivalently, $G$ is simple and $\{u,v\}$ forms a cut in $G$).
	The proof is a relatively straightforward but illustrative demonstration of our techniques.

	\begin{lemma} \label{lem:baltighttheta}
	    Let $(G,e)$ be a simple rooted $\theta$-chain, and let $C_1,C_2$ denote the two chains of $(G,e)$. Assume that Theorem \ref{thm:main} holds for graphs with fewer than $n(G)$ vertices.
	    Then
\begin{enumerate}
    \item [(i)] $\del(G,e) = -\frac{1}{2} + \min_{i \in [2]}\left(\del(\cl{C_i},e_{C_i}) + \bdel(\cl{C_{3-i}},e_{C_{3-i}})\right) \leq -\frac{1}{2}$, with equality if and only if $(G,e)$ is a balanced tight rooted $\theta$-chain,
    \item [(ii)]  $\bdel(G,e) \leq \frac{3}{2}  +\del(\cl{C_1},e_{C_1}) + \del(\cl{C_2},e_{C_2})\leq \frac{1}{2}$,
    \item [(iii)] $(\del(G,e),\bdel(G,e))=(-\frac{1}{2},\frac{1}{2})$ if and only if $(G,e)$ is a minimal rooted $\theta$-chain, and
    \item [(iv)] if $\del(G,e)=-1$ then $\min_{i\in[2]}\left(\del(\cl{C_i},e_{C_i})+\bdel(\cl{C_{3-i}},e_{C_{3-i}})\right) = -\frac{1}{2}$.
\end{enumerate}

	\end{lemma}

	\begin{proof}
    An even cover $F\in \ec(G,e)$ corresponds to a pair $(F_1,F_2)$ where $F_i \in \ec(\cl{C_i})$ for each $i\in[2]$ and $F_i \in \ec(\cl{C_i},e_{C_i})$ for exactly one $i \in[2]$.
    This correspondence gives $\exc(F) = \exc(F_1)+\exc(F_2)$. Since $n(G) = n(\cl{C_1})+n(\cl{C_2})+2$ and $n_2(G)=n_2(\cl{C_1})+n_2(\cl{C_2})$, we have
    \begin{align*}
        \minexc(G,e)
        &= \min_{\substack{F \in \ec(G,e)}} \exc(F) - 2 \\
        &= \min_{i\in[2]} \left(\min_{\substack{F_i\in\ec(\cl{C_i},e_{C_i})}} (\exc(F_i) - 2) + \min_{\substack{F_{3-i}\in\bec(\cl{C_{3-i}},e_{C_{3-i}})}} \exc(F_{3-i})\right) \\
        &= \min_{i\in[2]}\left(\minexc(\cl{C_i},e_{C_i})+\bminexc(\cl{C_{3-i}},e_{C_{3-i}})\right) \\
        &= \min_{i\in[2]}\left(\frac{n(\cl{C_i})+n_2(\cl{C_i})}{4}+\del(\cl{C_i},e_{C_i}) + \frac{n(\cl{C_{3-i}})+n_2(\cl{C_{3-i}})}{4}+\bdel(\cl{C_{3-i}},e_{C_{3-i}})\right) \\
        &= \min_{i\in[2]}\left(\frac{n(G)+n_2(G)}{4} - \frac{1}{2} + \del(\cl{C_i},e_{C_i})+\bdel(\cl{C_{3-i}},e_{C_{3-i}})\right).
    \end{align*}
    Therefore,
    \begin{align}
        \del(G,e)
        &=-\frac{1}{2}+ \min_{i\in[2]}\left( \del(\cl{C_i},e_{C_i})+\bdel(\cl{C_{3-i}},e_{C_{3-i}})\right), \label{eqn:uvcutdel=minchains}
    \end{align}
    whence for $i\in[2]$,

        \begin{align}
            \del(G,e)\leq -\frac{1}{2}+ \del(\cl{C_i},e_{C_i})+\bdel(\cl{C_{3-i}},e_{C_{3-i}}).  \label{eqn:uvcutdelGeCi}
    \end{align}
    By assumption, Theorem \ref{thm:main} holds for $(\cl{C_i},e_{C_i})$; so $\del(\cl{C_i},e_{C_i})+\bdel(\cl{C_i},e_{C_i}) \leq 0$ for each $i\in[2]$. Adding the two inequalities of \eqref{eqn:uvcutdelGeCi} gives
    \begin{align*}
        2\del(G,e)
        &\leq -1 + \sum_{i\in[2]} \left(\del(\cl{C_i},e_{C_i})+\bdel(\cl{C_i},e_{C_i})\right) \leq -1.
    \end{align*}
    Hence,
    \begin{equation}
        \del(G,e)\leq -\frac{1}{2}. \label{eqn:uvcutdel1/2}
    \end{equation}
    Moreover, $\del(G,e) = -\frac{1}{2}$ if and only if all of the above inequalities are tight, which means $(\cl{C_1},e_{C_1})$ and $(\cl{C_2},e_{C_2})$ are tight, and $$0 = \del(\cl{C_1},e_{C_1})+\bdel(\cl{C_2},e_{C_2}) = \del(\cl{C_1},e_{C_1})-\del(\cl{C_2},e_{C_2}).$$
    In other words, $C_1,C_2$ are balanced.
    Together with \eqref{eqn:uvcutdel=minchains} and \eqref{eqn:uvcutdel1/2}, this proves (i).

    If $F_i\in \ec(\cl{C_i}, e_{C_i})$ for each $i\in[2]$ then, by merging the cycles in $F_i$ containing $e_{C_i}$ for $i\in [2]$, we obtain an even cover $F\in \bec(G,e)$  with $\exc(F) = \exc(F_1)+\exc(F_2)-2$. So
    \begin{align*}
        \bminexc(G,e)
        &\leq \min_{\substack{F \in \bec(G,e)}} \exc(F) \\
        &\leq \min_{\substack{F_1\in\ec(\cl{C_1},e_{C_1})}} \exc(F_1) + \min_{\substack{F_2\in\ec(\cl{C_2},e_{C_2})}} (\exc(F_2) - 2) \\
        &= (\minexc(\cl{C_1},e_{C_1})+2)+\minexc(\cl{C_2},e_{C_2}) \\
        &= \frac{n(\cl{C_1})+n_2(\cl{C_1})}{4} + \del(\cl{C_1},e_{C_1}) + \frac{n(\cl{C_2})+n_2(\cl{C_2})}{4} + \del(\cl{C_2},e_{C_2}) + 2\\
        &= \frac{n(G)+n_2(G)}{4} + \frac{3}{2} + \del(\cl{C_1},e_{C_1})+\del(\cl{C_2},e_{C_2}).
    \end{align*}
    Hence,
    \begin{align*}
        \bdel(G,e)
        &\leq \frac{3}{2} + \del(\cl{C_1},e_{C_1})+\del(\cl{C_2},e_{C_2}).
    \end{align*}
    Since $G$ is simple, each $C_i$ is a nontrivial chain; so $\del(\cl{C_i},e_{C_i})\leq -\frac{1}{2}$ by the assumption that Theorem \ref{thm:main} holds for $(\cl{C_i},e_{C_i})$.
    This gives $\bdel(G,e)\leq \frac{1}{2}$ and proves (ii).

    To prove (iii), suppose $(\del(G,e), \bdel(G,e))=(-\frac{1}{2}, \frac{1}{2})$. Then  $\del(\cl{C_1},e_{C_1})+\del(\cl{C_2},e_{C_2}) =-1$ by (ii). Since $\del(\cl{C_i},e_{C_i})\le -\frac{1}{2}$ for $i\in [2]$ (by assumption),  $\del(\cl{C_i},e_{C_i})=-\frac{1}{2}$ for each $i\in[2]$.
    Moreover, each $(\cl{C_i},e_{C_i})$ is tight (by (i)), so $(G,e)$ is a minimal rooted $\theta$-chain.

    Finally, note that (iv) follows from (i).
	\end{proof}

	The next lemma says that given a choice of adding an edge $uv_1$ or $uv_2$ to a 2-connected subcubic graph $Z$, the two quantities $\del(Z+uv_1,uv_1)$ and $\del(Z+uv_2,uv_2)$ cannot both be large.

	\begin{lemma} \label{lem:block3vx}
		Let $Z$ be a 2-connected simple subcubic graph and let $u,v_1,v_2$ be three distinct vertices of degree 2 in $Z$. Assume Theorem \ref{thm:main} holds for graphs with at most $n(Z)$ vertices. Then $\del(Z+uv_1,uv_1) + \del(Z+uv_2,uv_2) \leq -2$.
	\end{lemma}

	\begin{proof}
	    By the assumption that Theorem \ref{thm:main} holds for graphs with at most $n(Z)$ vertices, we have $\del(Z+uv_i,uv_i)\leq -\frac{1}{2}$ for each $i\in[2]$, with equality if and only if $(Z+uv_i,uv_i)$ is a balanced tight rooted $\theta$-chain.
		If both $\del(Z+uv_1,uv_1)\le -1$ and $\del(Z+uv_2,uv_2) \leq -1$, then there is nothing to prove.
		So we may assume by symmetry that $\del(Z+uv_1,uv_1) = -\frac{1}{2}$; thus $(Z+uv_1,uv_1)$ is a balanced tight rooted $\theta$-chain. Note that it suffices to show that $\del(Z+uv_2,uv_2)\leq -\frac{3}{2}$.

		Let $C_1,C_2$ denote the two chains of $(Z+uv_1,uv_1)$.
		Let us assume without loss of generality that $v_2 \in V(C_1)$.
		Write $C_1 = v_1e_0B_1e_1B_2\dots B_{k}e_{k}u$ (where $k\geq 1$) and write its chain-blocks $(\cl{B_i}, \cl{e_i})$ for all $i\in[k]$.
		Since $C_1,C_2$ are balanced, we have $\del(\cl{C_1},e_{C_1}) = \del(\cl{C_2},e_{C_2})$, and since they are both tight, we have $\del(\cl{C_i},e_{C_i})+\bdel(\cl{C_i},e_{C_i})=0$ for $i\in [2]$. So by Proposition \ref{prop:tight} and the assumption that Theorem \ref{thm:main} holds for each $(\cl{B_i},\cl{e_i})$, we have
		\begin{equation}\label{eqn:lem:block3vxtightblocks}
		    \del(\cl{B_i}, \cl{e_i}) + \bdel(\cl{B_i},\cl{e_i}) = 0 \qquad \text{for all $i\in[k]$.}
		\end{equation}
		Let $\ell\in[k]$ be the unique index such that $v_2 \in B_\ell$. (Note $\ell$ is well defined as $Z$ is subcubic and $v_2$ has degree 2 in $Z$.)
		Let $v'$ denote the vertex of $B_\ell$ incident with $e_{\ell-1}$.

		Then there is an even cover $F\in \ec(Z+uv_2, uv_2)$ obtained from a tuple $(F',F_1,\dots,F_k)$ where
		$F'\in \ec(\cl{C_2},e_{C_2})$, $F_i\in \ec(\cl{B_i},\cl{e_i})$ for each $i\in[\ell-1]$, $F_\ell\in \ec(B_\ell+v'v_2,v'v_2)$, and $F_j\in\bec(\cl{B_j},\cl{e_j})$ for each $j=\ell+1,\dots,k$.
		This gives $\exc(F)-2 = (\exc(F')-2)+ \sum_{i=1}^\ell (\exc(F_i)-2) + \sum_{j=\ell+1}^k \exc(F_j)$.
		Moreover, since $n(B_\ell+v'v_2)=n(\cl{B_\ell})$ and $n_2(B_\ell+v'v_2)=n_2(\cl{B_\ell})$, we have
		\begin{align*}
		    n(Z+uv_2) &= 2+n(\cl{C_2})+\sum_{i=1}^{\ell-1}n(\cl{B_i})+n(B_\ell+v'v_2)+\sum_{j=\ell+1}^k n(\cl{B_j}), \\
		    n_2(Z+uv_2) &= n_2(\cl{C_2})+\sum_{i=1}^{\ell-1}n_2(\cl{B_i})+n_2(B_\ell+v'v_2)+\sum_{j=\ell+1}^k n_2(\cl{B_j}).
		\end{align*}
		This gives
		\begin{align*}
			\minexc(Z+uv_2,uv_2)
			&\leq \minexc(\cl{C_2},e_{C_2}) + \sum_{i=1}^{\ell - 1} \minexc(\cl{B_i}, \cl{e_i}) + \minexc(B_\ell+v'v_2,v'v_2)+ \sum_{j=\ell+1}^k \bminexc(\cl{B_j},\cl{e_j}) \\
			&= \frac{n(Z+uv_2)+n_2(Z+uv_2)}{4} - \frac{1}{2} + \del(\cl{C_2},e_{C_2}) + \sum_{i=1}^{\ell-1}\del(\cl{B_i}, \cl{e_i})\\
			&\qquad +\del(B_\ell+v'v_2,v'v_2)  + \sum_{j=\ell+1}^k \bdel(\cl{B_j},\cl{e_j}),
		\end{align*}
		whence
		\begin{align*}
		    \del(Z+uv_2,uv_2)
			&\leq -\frac{1}{2}+ \del(\cl{C_2},e_{C_2}) + \sum_{i=1}^{\ell - 1} \del(\cl{B_i}, \cl{e_i}) + \del(B_\ell+v'v_2,v'v_2) + \sum_{j=\ell+1}^k \bdel(\cl{B_j},\cl{e_j}).
		\end{align*}

		Note that $\bminexc(\cl{B_\ell},\cl{e_\ell}) = \bminexc(B_\ell+v'v_2,v'v_2)$ since both quantities are equal to the minimum excess of an even cover of $B_\ell$. This implies $\bdel(\cl{B_\ell},\cl{e_\ell})=\bdel(B_\ell+v'v_2,v'v_2)$.
		Using (\ref{eqn:lem:block3vxtightblocks}) and that $\del(B_\ell+v'v_2,v'v_2)+\bdel(B_\ell+v'v_2,v'v_2)\leq 0$ as Theorem \ref{thm:main} holds for $(B_\ell+v'v_2,v'v_2)$ (by assumption), we have
		\begin{align*}
			\del(Z+uv_2,uv_2)
			&\leq -\frac{1}{2}+\del(\cl{C_2},e_{C_2}) + \sum_{i=1}^{\ell - 1} \big(-\bdel(\cl{B_i}, \cl{e_i})\big) + \big(-\bdel(B_\ell+v'v_2,v'v_2)\big) + \sum_{j=\ell+1}^k \bdel(\cl{B_j},\cl{e_j}) \\
			&= -\frac{1}{2}+ \del(\cl{C_2},e_{C_2}) + \sum_{i=1}^{\ell - 1} \big(-\bdel(\cl{B_i}, \cl{e_i})\big) + \big(-\bdel(\cl{B_\ell}, \cl{e_\ell})\big) + \sum_{j=\ell+1}^k \bdel(\cl{B_j},\cl{e_j}) \\
			&= -\frac{1}{2}+ \del(\cl{C_2},e_{C_2}) + \sum_{j=1}^{k} \bdel(\cl{B_j},\cl{e_j}) - 2\sum_{j=1}^\ell \bdel(\cl{B_j}, \cl{e_j}) \\
			&= -\frac{1}{2}+\del(\cl{C_2},e_{C_2})+\bdel(\cl{C_1},e_{C_1}) - 2\sum_{j=1}^\ell \bdel(\cl{B_j},\cl{e_j}) \tag{by Proposition \ref{prop:subcubicchains}}\\
			&= -\frac{1}{2} - 2\sum_{j=1}^\ell \bdel(\cl{B_j}, \cl{e_j}) \tag{as $C_1$ and $C_2$ are balanced and tight} \\
			&\leq -\frac{3}{2},
        \end{align*}
        since $-\bdel(\cl{B_j},\cl{e_j})=\del(\cl{B_j}, \cl{e_j})\le -1/2$ for all $j\in[k]$ by \eqref{eqn:lem:block3vxtightblocks} and the assumption that Theorem \ref{thm:main} holds for $(\cl{B_j},\cl{e_j})$.
	\end{proof}

	We can now prove the following lemma for $\theta$-chains.

    \begin{lemma} \label{lem:minimaltheta2}
	    Let $G$ be a subcubic graph with $e=uv\in E(G)$ such that $G-e$ is simple and 2-connected.
	    Assume that Theorem \ref{thm:main} holds for graphs with fewer than $n(G)$ vertices.
	    Let $G_u$ be the graph obtained from $G-e$ by suppressing $u$ into an edge $f_u$, and assume that $(G_u,f_u)$ is a rooted $\theta$-chain. Then
	    \begin{enumerate}[label=(\roman*)]
	        \item $\bdel(G,e) \le \frac{3}{2}$, with equality if and only if $G-e$ is a minimal $\theta$-chain whose three minimal chains can be chosen to have common endpoints $N(u)\setminus \{v\}$,
	        \item $\del(G,e) \le -\frac{3}{2}$, and
	        \item $(\del(G,e),\bdel(G,e)) = (-\frac{3}{2}, \frac{3}{2})$ if and only if $G-e$ is a minimal $\theta$-chain and $e$ joins two nonadjacent vertices of a 4-cycle in $G-e$.
	    \end{enumerate}

	\end{lemma}

	\begin{proof}
	Let $N(u)\setminus \{v\}=\{x,y\}$, the set of endpoints of $f_u$.
	Let $C_1,C_2$ denote the two chains of $(G_u,f_u)$ with common endpoints $\{x,y\}$, and let $C_3$ denote the subcubic chain $x(xu)u(uy)y$.
    Note that $n(G) = 2+\sum_{i=1}^3 n(\cl{C_i})$, $n_2(G) = -2+\sum_{i=1}^3 n_2(\cl{C_i})$ (since the $\cl{C_i}$'s do not account for the edge $e$), and $\cl{C_3}$ is a loop. Let $i_1,i_2,i_3$ be a permutation of $[3]$ such that $\del(\cl{C_{i_1}},e_{C_{i_1}}) \le \del(\cl{C_{i_2}},e_{C_{i_2}}) \le \del(\cl{C_{i_3}},e_{C_{i_3}})$.

     Consider a triple $(F_1,F_2,F_3)$ such that $F_{i_1} \in \ec(\cl{C_{i_1}},e_{C_{i_1}})$, $F_{i_2} \in \ec(\cl{C_{i_2}},e_{C_{i_2}})$, and $F_{i_3} \in \bec(\cl{C_{i_3}},e_{C_{i_3}})$. Let $F\in \bec(G,e)$ be obtained from $F_1\cup F_2\cup F_3$ by merging the cycles in $F_{i_1}$, $F_{i_2}$ through $e_{C_{i_1}}$, $e_{C_{i_2}}$. Then $\exc(F) -2 = (\exc(F_{i_1})-2)+(\exc(F_{i_2})-2)+\exc(F_{i_3})$; so
    \begin{align*}
        \bminexc(G,e) - 2
        &\leq \minexc(\cl{C_{i_1}},e_{C_{i_1}})+\minexc(\cl{C_{i_2}},e_{C_{i_2}})+\bminexc(\cl{C_{i_3}},e_{C_{i_3}}) \\
        &= \frac{n(G)+n_2(G)}{4} + \del(\cl{C_{i_1}},e_{C_{i_1}})+\del(\cl{C_{i_2}},e_{C_{i_2}})+\bdel(\cl{C_{i_3}},e_{C_{i_3}}).
    \end{align*}
    Since Theorem \ref{thm:main} holds for $(\cl{C_i},e_{C_i})$ for each $i\in[3]$ (by assumption), we have $\bdel(\cl{C_{i_3}},e_{C_{i_3}}) \leq -\del(\cl{C_{i_3}},e_{C_{i_3}})\leq -\del(\cl{C_{i_2}},e_{C_{i_2}})$ and $\del(\cl{C_{i}},e_{C_{i}}) \le -\frac{1}{2}$ for $i\in [3]$, which gives
    \begin{align*}
        \bminexc(G,e)-2 \leq \frac{n(G)+n_2(G)}{4}+\del(\cl{C_{i_1}},e_{C_{i_1}}) \leq \frac{n(G)+n_2(G)}{4}-\frac{1}{2}.
    \end{align*}
    Therefore, $\bminexc(G,e)\leq \frac{n(G)+n_2(G)}{4}+\frac{3}{2}$, and $\bdel(G,e)\leq \frac{3}{2}$.

    Suppose $\bdel(G,e)= \frac{3}{2}$. Then the above inequalities hold with equality. Hence,  $-\frac{1}{2} = \del(\cl{C_{i_1}},e_{C_{i_1}})=\del(\cl{C_{i_2}},e_{C_{i_2}})=\del(\cl{C_{i_3}},e_{C_{i_3}})$. Since Theorem \ref{thm:main} holds for all $(\cl{C_i},e_{C_i})$ (by assumption),
    $(\cl{C_i},e_{C_i})$ is tight (hence minimal) for all $i\in[3]$.
    Therefore, $G-e$ is a minimal $\theta$-chain with its three chains having common endpoints $N(u)\setminus \{v\}$.

    Now suppose $G-e$ is a minimal $\theta$-chain with the three minimal chains $C_1,C_2, C_3$ with common endpoints $N(u)\setminus \{v\}$. Let $F\in \bec(G,e)$. If $F$ contains a cycle through two of $C_1,C_2,C_3$, then the above argument shows $\exc(F) \geq \frac{n(G)+n_2(G)}{4}+\frac{3}{2}$.
    So we just need to show that if $F$ does not contain a cycle through any of $C_1,C_2,C_3$, then $\exc(F)\geq \frac{n(G)+n_2(G)}{4} + \frac{3}{2}$.
    Indeed, such $F$ when restricted to $(\cl{C_i},e_{C_i})$ for $i\in [3]$ gives a triple $(F_1,F_2,F_3)$ such that $F_i \in \bec(\cl{C_i},e_{C_i})$ for each $i\in[3]$, and $\exc(F) = 2+\sum_{i=1}^3\exc(F_i)$ (since the two vertices of $N(u)\setminus \{v\}$ are isolated in $F$).
    So
    \begin{align*}
        \exc(F) &\geq 2 + \sum_{i=1}^3 \bminexc(\cl{C_i},e_{C_i}) \\
        &= 2+\sum_{i=1}^3 \left(\frac{n(\cl{C_i})+n_2(\cl{C_i})}{4} + \bdel(\cl{C_i},e_{C_i})\right) \\
        &= \frac{n(G)+n_2(G)}{4}+2 + \sum_{i=1}^3 \bdel(\cl{C_i},e_{C_i}) \\
        &= \frac{n(G)+n_2(G)}{4}+ \frac{7}{2}.
    \end{align*}
    The last equality holds since $\bdel(\cl{C_i},e_{C_i})=\frac{1}{2}$ for each $i\in[3]$, completing the proof of (i).

    We now prove (ii) and (iii).
    Let us assume without loss of generality that $v\in V(C_1)$, and write $C_1 = xe_0B_1e_1B_2\dots B_ke_ky$ with chain-blocks $(\cl{B_i},\cl{e_i})$.
    Let $\ell\in[k]$ denote the unique index such that $v\in V(B_\ell)$.
    By symmetry, we may assume that $\sum_{i=1}^{\ell-1} \del(\cl{B_i},\cl{e_i}) \leq \sum_{j=\ell+1}^k \del(\cl{B_j},\cl{e_j})$.
    Then, by the assumption that Theorem \ref{thm:main} holds for each $(\cl{B_j},\cl{e_j})$, we have
    \begin{equation} \label{eqn:lem:Gufutheta}
    \sum_{j=\ell+1}^k \bdel(\cl{B_j},\cl{e_j}) \leq \sum_{j=\ell+1}^k (-\del(\cl{B_j},\cl{e_j})) \leq -\left(\sum_{i=1}^{\ell-1} \del(\cl{B_i},\cl{e_i})\right).
    \end{equation}

    Consider the tuple of even covers $(F_1,\dots,F_k,F^2)$, where $F_i \in \ec(\cl{B_i},\cl{e_i})$ for $i\in[\ell-1]$, $F_\ell\in\ec(B_\ell+x'v,x'v)$ where $x'$ is the endpoint of $e_{\ell-1}$ in $B_{\ell}$, $F_j \in \bec(\cl{B_j},\cl{e_j})$ for $j=\ell+1,\dots,k$, and $F^2\in\ec(\cl{C_2},e_{C_2})$.
    This corresponds to an even cover $F\in\ec(G,e)$ containing a cycle through all of $xe_0B_1\dots B_{\ell-1}e_{\ell-1}$, $e$, $uy$, and $C_2$, such that
    \begin{align*}
        \exc(F)-2 = \sum_{i=1}^{\ell-1} (\exc(F_i)-2) + (\exc(F_\ell)-2) + \sum_{j=\ell+1}^k \exc(F_j) + (\exc(F^2)-2).
    \end{align*}
    Since
    \begin{align*}
        n(G) &= \sum_{i=1}^{\ell-1} n(\cl{B_i}) + n(B_\ell+x'v)+\sum_{j=\ell+1}^k n(\cl{B_j}) + n(\cl{C_2}) +3, \text{ and} \\
        n_2(G) &=  \sum_{i=1}^{\ell-1} n_2(\cl{B_i}) + n_2(B_\ell+x'v)+\sum_{j=\ell+1}^k n_2(\cl{B_j}) + n_2(\cl{C_2}) -1,
    \end{align*}
    we have
    \begin{align*}
        \minexc(G,e)
        &\leq \sum_{i=1}^{\ell-1}\minexc(\cl{B_i},\cl{e_i}) + \minexc(B_\ell+x'v,x'v) + \sum_{j=\ell+1}^k \bminexc(\cl{B_j},\cl{e_j}) + \minexc(\cl{C_2},e_{C_2}) \\
        &= \frac{n(G)+n_2(G)}{4} - \frac{1}{2} + \left(\sum_{i=1}^{\ell-1}\del(\cl{B_i},\cl{e_i})
        \right)+ \del(B_\ell+x'v,x'v) + \left(\sum_{j=\ell+1}^k \bdel(\cl{B_j},\cl{e_j})\right) + \del(\cl{C_2},e_{C_2}) \\
        &\leq \frac{n(G)+n_2(G)}{4} - \frac{1}{2} + \del(B_\ell+x'v,x'v)  + \del(\cl{C_2},e_{C_2})  \tag{by \eqref{eqn:lem:Gufutheta}} \\
        &\leq \frac{n(G)+n_2(G)}{4} - \frac{3}{2},
    \end{align*}
    where the last inequality follows as by our assumption Theorem \ref{thm:main} holds for $(B_\ell+x'v,x'v)$ and $(\cl{C_2},e_{C_2})$.
    Hence $\del(G,e)\leq -\frac{3}{2}$ and (ii) holds.

    To prove (iii), suppose $(\del(G,e),\bdel(G,e)) =( -\frac{3}{2}, \frac{3}{2})$. Then equality holds above, so we have $\del(B_\ell+x'v,x'v) = \del(\cl{C_2},e_{C_2}) =-\frac{1}{2}$.
    Moreover, $C_1$ and $C_2$ are minimal chains (by (i)), which implies $k=\ell=1$ and $\del(\cl{B_\ell},\cl{e_\ell}) = \del(\cl{C_1},e_{C_1}) = -\frac{1}{2}$ (by Proposition \ref{prop:subcubicchains}).
    So $\del(\cl{B_\ell},\cl{e_{\ell}}) +\del(B_\ell+x'v,x'v)=-1$.
    Now $B_\ell$ is a single vertex; otherwise, by applying Lemma \ref{lem:block3vx} to $B_\ell$, $x'$, the other endpoint $y'$ of $\cl{e_\ell}$, and $v$, we obtain $\del(\cl{B_\ell},\cl{e_{\ell}}) +\del(B_\ell+x'v,x'v) =\del(B_\ell+x'y',x'y') +\del(B_\ell+x'v,x'v) \le -2$, a contradiction.
    Therefore, we have $B_\ell = \{x'\}=\{v\}$, and $e$ joins two nonadjacent vertices of the 4-cycle $xvyux$.
    \end{proof}

    We conclude this section with a lemma bounding $\bdel(G,e)$, which proves statement \ref{main3/2} of Theorem \ref{thm:main}, assuming Theorem \ref{thm:main} for smaller graphs.
    \begin{lemma} \label{lem:extremalG-e}
    Let $G$ be a 2-connected subcubic graph with $e=uv\in E(G)$ such that $G-e$ is simple and 2-connected.
    Assume that Theorem \ref{thm:main} holds for graphs with fewer than $n(G)$ vertices.
    Then
    $\bdel(G,e) \leq \frac{3}{2}$, with equality if and only if $(G_u,f_u)$ is a minimal rooted $\theta$-chain, where $G_u$ is the graph obtained from $G - e$ by suppressing $u$ into an edge $f_u$.
    \end{lemma}
    \begin{proof}
    Since $G-e$ is 2-connected, both $u$ and $v$ have degrees 3.
    Define $G_u,f_u$ as stated in the lemma.
    We claim that
    \begin{align}\label{eqn:delbdelGufu}
        \bdel(G,e) = \min\{\del(G_u,f_u)+2, \bdel(G_u,f_u)+1\}.
    \end{align}
    Indeed, there is a bijective correspondence between $\bec(G,e)$ and $\ec(G_u)$ obtained as follows.
    If $F\in\bec(G,e)$ contains a cycle through $u$, then we obtain $F_u \in \ec(G_u,f_u)$ by suppressing $u$ in $F$, and we have $\exc(F) = \exc(F_u)$.
    Otherwise, if $u$ is an isolated vertex in $F$, then we obtain $F_u\in \bec(G_u,f_u)$ by removing $u$ from $F$, and we have $\exc(F) = \exc(F_u)+1$.
    Since $n(G)+n_2(G) = n(G_u)+n_2(G_u)$, \eqref{eqn:delbdelGufu} follows from the definitions of $\del,\bdel$.

    It follows from \eqref{eqn:delbdelGufu} that $\bdel(G,e)\leq \del(G_u,f_u)+2\leq \frac{3}{2}$ by the assumption that Theorem \ref{thm:main} holds for $(G_u,f_u)$.
    Moreover, $\bdel(G,e)=\frac{3}{2}$ if and only if $\del(G_u,f_u)=-\frac{1}{2}$ and $\bdel(G_u,f_u)=\frac{1}{2}$, which is equivalent to $(G_u,f_u)$ being a minimal rooted $\theta$-chain by Lemma \ref{lem:baltighttheta}.
    \end{proof}

	\section{Proof of Theorem~\ref{thm:main}}\label{sec:mainproof}

		We proceed by induction on $n(G)$.
		If $G$ is a loop (respectively, a digon), then $(\del(G,e),\bdel(G,e))=(-\frac{1}{2},\frac{1}{2})$ (respectively, $(-1,1)$) and the theorem holds; so assume that $G$ is neither a loop nor a digon.
		Note that \ref{maindelbdel} is implied by \ref{main1/2}, \ref{main3/2}, and the inductive hypothesis. Indeed, if $G-e$ is not 2-connected, then $(G,e)$ is the closure of a subcubic chain with at least two chain-blocks, and \ref{maindelbdel} follows from Proposition~\ref{prop:tight} and the inductive hypothesis applied to these chain-blocks. Suppose $G-e$ is 2-connected. If $\del(G,e)\leq -1$ and $\bdel(G,e)\leq 1$, then \ref{maindelbdel} holds.
		Otherwise, we have $\del(G,e)=-\frac{1}{2}$ or $\bdel(G,e)=\frac{3}{2}$.
		In the former case, \ref{maindelbdel} follows from \ref{main1/2} and Lemma \ref{lem:baltighttheta}; in the latter case, \ref{maindelbdel} follows from Lemmas \ref{lem:extremalG-e} and \ref{lem:minimaltheta2}.
	    Also note that Lemmas \ref{lem:minimaltheta2} and \ref{lem:extremalG-e} imply \ref{main3/2}.
		Therefore, it suffices to prove \ref{main1/2} and \ref{main1}.

		If $G-\{u,v\}$ is disconnected, then  \ref{main1/2} and \ref{main1} both hold by Lemma \ref{lem:baltighttheta}.
		So we may assume that  $G-\{u,v\}$ is connected.
		It now suffices to show that $\del(G,e)\leq -1$ and that if equality holds, then one of the outcomes of \ref{main1} holds.

		\begin{claim}
		We may assume that $G$ is simple.
		\end{claim}
		\begin{subproof}
		Since $G-e$ is simple, if $G$ is not simple, then there is exactly one edge $e^*$ parallel with $e$.
		Let $G'$ be the graph obtained from $G$ by suppressing  $\{u,v\}$ to an edge $e'$.

				Then $n(G) = n(G')+2$ and $n_2(G) = n_2(G')$.
		By the inductive hypothesis, we have $\del(G',e')\leq -\frac{1}{2}$.
		But every even cover $F'\in \ec(G',e')$ gives an even cover $F\in \ec(G,e)$ with the same excess, so
		\begin{align*}
		    \del(G,e)
		    &= \min_{\substack{F\in \ec(G,e)}} \exc(F) -2 - \frac{n(G)+n_2(G)}{4} \\
		    &\leq \min_{\substack{F'\in \ec(G',e')}} \exc(F') -2 - \frac{n(G')+n_2(G')+2}{4} \\
		    &= \del(G',e') - \frac{1}{2} \\
		    &\leq -1.
		\end{align*}

		Now suppose $\del(G,e)=-1$. Then both inequalities above are tight; in particular, we have $\del(G',e')=-\frac{1}{2}$, and by the inductive hypothesis, $G'$ is a loop or  $(G',e')$ is a balanced tight rooted $\theta$-chain.
		If $G'$ is a loop then $(G,e)$ satisfies (b) of \ref{main1}. So assume that $(G',e')$ is a balanced tight rooted $\theta$-chain, and let $C_1,C_2$ denote the two chains of $(G',e')$.

		Then a pair of even covers $F_1,F_2$ where $F_i\in \ec(\cl{C_i},e_{C_i})$ for each $i\in[2]$ gives an even cover $F \in \ec(G,e)$ by combining the two cycles of $F_i$ through $e_{C_i}$ and adding the cycle with edge set $\{e,e^*\}$, with $$\exc(F)-2 = (\exc(F_1)-2)+(\exc(F_2)-2)+2.$$
		Since $n(G) = n(\cl{C_1})+n(\cl{C_2})+4$ and $n_2(G) = n_2(\cl{C_1})+n_2(\cl{C_2})$, we have
		\begin{align*}
		    \minexc(G,e)
		    &\leq \minexc(\cl{C_1},e_{C_1})+\minexc(\cl{C_2},e_{C_2}) + 2 \\
		    &= \frac{n(G)+n_2(G)}{4} + 1 + \del(\cl{C_1},e_{C_1}) + \del(\cl{C_2},e_{C_2}),
		\end{align*}
		so $\del(G,e)\leq 1+\del(\cl{C_1},e_{C_1})+\del(\cl{C_2},e_{C_2})$.
		Thus, we have $\del(\cl{C_i},e_{C_i})\in\{-\frac{1}{2},-1\}$ for each $i\in[2]$; in other words, $(G',e')$ is a near-minimal rooted $\theta$-chain. So $(G,e)$ satisfies (b) of \ref{main1}.
		\end{subproof}

    \begin{claim}\label{clm:no2edgecut}
		We may assume that $e$ is not in any 2-edge-cut of $G$.
		\end{claim}
		\begin{subproof}
		Suppose there is an edge $e'$ such that $\{e,e'\}$ is a 2-edge-cut of $G$. Let $C$ be a subcubic chain of $G$ with end edges $e,e'$.
		By Proposition \ref{prop:contractchain} and by the inductive hypothesis applied to $(G/C,e_{G/C})$ and $(\cl C,e_C)$, we have
		\begin{align*}
		    \del(G,e) &= \del(G/C,e_{G/C})+\del(\cl C,e_C) \leq -1.
		\end{align*}
		Moreover, if $\del(G,e)=-1$, then
		$\del(G/C,e_{G/C})=\del(\cl C,e_C) = -\frac{1}{2}$, so $(G/C,e_{G/C})$ and $(\cl C,e_C)$ are loops or balanced tight rooted $\theta$-chains and  (c) of \ref{main1} holds for $(G,e)$.
		\end{subproof}
    By Claim \ref{clm:no2edgecut},
    let $u_1,u_2$ denote the two neighbors of $u$ distinct from $v$, and let $v_1,v_2$ denote the two neighbors of $v$ distinct from $u$.
		Moreover, there exist two disjoint paths $P_1,P_2$ from $\{u_1,u_2\}$ to $\{v_1,v_2\}$ in $G-\{u,v\}$.
		We may assume without loss of generality that the set of endpoints of $P_i$ is $\{u_i,v_i\}$, $i \in [2]$.

		Let $S$ denote the set of all cut edges in $G-\{u,v\}$. Then each component of $G-\{u,v\}-S$ is either an isolated vertex or 2-connected.
		\begin{claim}\label{clm:Z_1Z_2}
		For each $i \in [2]$, there is a unique component $Z_i$ of $G-\{u,v\}-S$,  such that there are three paths in $G-\{u,v\}$ from $Z_i$ to $\{u_i,v_i,u_{3-i}\}$, pairwise disjoint except possibly at their endpoints in $Z_i$, and there are three paths in $G-\{u,v\}$ from $Z_i$ to $\{u_i,v_i,v_{3-i}\}$, pairwise disjoint except possibly at their endpoints in $Z_i$. See Figures \ref{fig:Z1neqZ2} and \ref{fig:Z1=Z2}.
		\end{claim}
		\begin{subproof}
		By symmetry, it suffices to prove the claim for $i=1$.
		First, we show that there is a unique component $Z_1$ of $G-\{u,v\}-S$ such that there are three paths in $G-\{u,v\}$ from $Z_1$ to $\{u_1,v_1,u_2\}$, pairwise disjoint except possibly at their endpoints in $Z_1$.
		Indeed, if there were two distinct such components $Z,Z'$, they are by definition separated by a cut-edge $s\in S$ of $G-\{u,v\}$.
		But $G-\{u,v\}-s$ has exactly two connected components, one of which contains at least two of $\{u_1,v_1,u_2\}$, so one of $Z,Z'$ is separated from two vertices of $\{u_1,v_1,u_2\}$ by a cut-edge, contradicting the assumptions on $Z,Z'$.

		Similarly, there is a unique connected component $Z_1'$ of $G-\{u,v\}-S$ such that there are three paths in $G-\{u,v\}$ from $Z_1'$ to $\{u_1,v_1,v_2\}$, pairwise disjoint except possibly at their endpoints in $Z_1'$.
		We now show that $Z_1=Z_1'$.
		Otherwise, there is a cut edge $s$ of $G-\{u,v\}$ separating $Z_1$ from $Z_1'$.
		Then the two connected components of $G-\{u,v\}-s$ each contain exactly one of $\{u_1,v_1\}$ and exactly one of $\{u_2,v_2\}$.
		But this implies that both $P_1$ and $P_2$ contain $s$, contradicting the assumption that $P_1$ and $P_2$ are disjoint.
		\end{subproof}

		There are two cases to consider: either $Z_1\neq Z_2$ or $Z_1 = Z_2$.
		For $i\in[2]$, let $u_i'$ (respectively, $v_i'$)  denote the vertex of $Z_i$ that is the endpoint of a (possibly trivial) path in $G-\{u,v\}$ from $u_i$ (respectively, $v_i$) to $Z_i$ that is internally disjoint from $Z_1 \cup Z_2$.
		Note that $u_i'$ and $v_i'$ are uniquely determined. For $i\in [2]$,
		let $U_i$ denote the unique (possibly trivial) subcubic chain of $G-\{v,u_{3-i}\}$ with endpoints $\{u,u_i'\}$, and let $V_i$ denote the unique subcubic chain of $G-\{u,v_{3-i}\}$ with endpoints $\{v,v_i'\}$.

		\begin{figure}[H]
		    \centering
		    \includegraphics[scale=0.35]{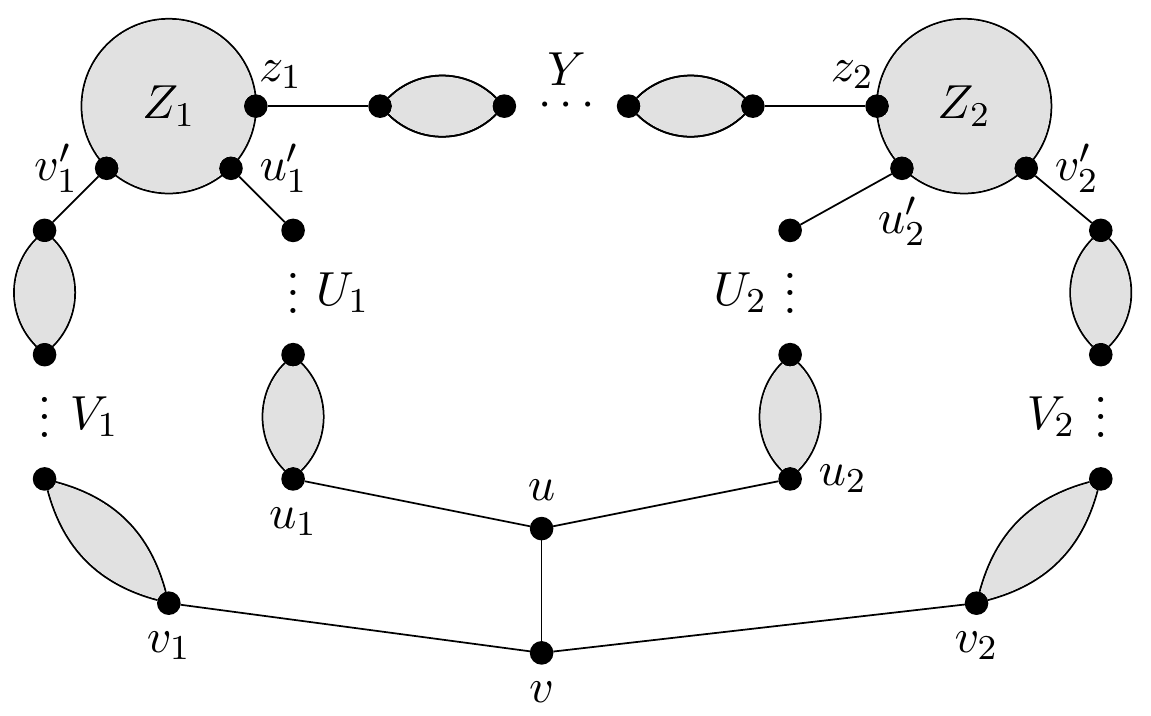}
		    \caption{$Z_1 \neq Z_2$}
		    \label{fig:Z1neqZ2}
		\end{figure}

		{\bf Case 1: $Z_1 \neq Z_2$}.

		There is a cut-edge separating $Z_1$ and $Z_2$ in $G-\{u,v\}$ and there is a unique subcubic chain $Y$ of $G-\{u,v\}$ with an endpoint $z_i\in Z_i$ for each $i \in [2]$, internally disjoint from $Z_1\cup Z_2$.
		Then $G$ is the union of $U_1, U_2, V_1, V_2, Z_1, Z_2, Y,$ and the edge $e=uv$.
		We have, for $w_1\in\{u_1',v_1'\}$ and $w_2\in\{u_2',v_2'\}$,
		\begin{align*}
			n(G) &=n(\cl{U_1})+n(\cl{U_2})+n(\cl{V_1})+n(\cl{V_2})+n(Z_1+w_1z_1) + n(Z_2+w_2z_2) + n(\cl Y) + 2, \\
			n_2(G) &= n_2(\cl{U_1})+n_2(\cl{U_2})+n_2(\cl{V_1})+n_2(\cl{V_2})+ n_2(Z_1+w_1z_1)+n_2(Z_2+w_2z_2)+ n_2(\cl Y)-2.
		\end{align*}

		Suppose $F\in \ec(G,e)$ goes through $U_1, Y$, and $V_2$. Then there is a correspondence between $F$ and the tuple $(F_{U_1}, F_{Z_1}, F_Y, F_{Z_2},F_{V_2}, F_{U_2},F_{V_1})$, where
		\begin{itemize}
		    \item $F_{U_1}\in \ec(\cl{U_1}, e_{U_1}),\ F_{Z_1}\in \ec(Z_1+u_1'z_1, u_1'z_1),\  F_Y\in \ec(\cl Y,e_Y),\ F_{Z_2}\in \ec(Z_2+v_2'z_2,v_2'z_2),\ F_{V_2}\in \ec(\cl{V_2},e_{V_2})$, and
		    \item $F_{U_2}\in \bec(\cl{U_2},e_{U_2}),\  F_{V_1}\in\bec(\cl{V_1},e_{V_1})$.
		\end{itemize}
		This gives
		\begin{align*}
			\minexc(G,e) &\leq \minexc(\cl{U_1}, e_{U_1}) + \minexc(Z_1+u'_1z_1,u'_1z_1) + \minexc(\cl Y, e_Y)  + \minexc(Z_{2}+v'_{2}z_{2},v'_{2}z_{2}) + \minexc(\cl{V_2},e_{V_{2}}) \\
			&\quad + \bminexc(\cl{U_2},e_{U_2}) + \bminexc(\cl{V_1},e_{V_1})  \\
			&= \frac{n(G)+n_2(G)}{4} + \del(\cl{U_1}, e_{U_1}) + \del(Z_1+u'_1z_1,u'_1z_1) + \del(\cl Y, e_Y) + \del(Z_{2}+v'_{2}z_{2},v'_{2}z_{2}) + \del(\cl{V_2},e_{V_{2}})  \\
			&\qquad\qquad
			+ \bdel(\cl{U_2},e_{U_2}) + \bdel(\cl{V_1},e_{V_1}),
		\end{align*}
		hence
		\begin{align}
		\label{eqn:case1del1}
		\begin{split}
			\del(G,e) &\leq
			\del(\cl{U_1}, e_{U_1}) + \del(Z_1+u'_1z_1,u'_1z_1) + \del(\cl Y, e_Y) + \del(Z_{2}+v'_{2}z_{2},v'_{2}z_{2}) + \del(\cl{V_2},e_{V_{2}}) \\
			&\quad
			+ \bdel(\cl{U_2},e_{U_2}) + \bdel(\cl{V_1},e_{V_1}).
		\end{split}
		\end{align}

		Similarly, by considering an even cover in  $\ec(G, e)$ through $U_2,Y$, and $V_1$, we obtain
		\begin{align}
		\label{eqn:case1del2}
		\begin{split}
			\del(G,e) &\leq
			\del(\cl{U_2}, e_{U_2}) + \del(Z_2+u'_2z_2,u'_2z_2) + \del(\cl Y, e_Y) + \del(Z_{1}+v'_{1}z_{1},v'_{1}z_{1}) + \del(\cl{V_1},e_{V_{1}}) \\
			&\quad
			+ \bdel(\cl{U_1},e_{U_1}) + \bdel(\cl{V_2},e_{V_2}).
		\end{split}
		\end{align}

		Now suppose $\del(G,e)\geq -1$. Then
		\begin{align*}
			\begin{split}
			-1
			&\leq \del(\cl{U_1}, e_{U_1}) + \del(Z_1+u'_1z_1,u'_1z_1) + \del(\cl Y, e_Y) + \del(Z_{2}+v'_{2}z_{2},v'_{2}z_{2}) + \del(\cl{V_2},e_{V_{2}}) \\
			&\quad + \bdel(\cl{U_2},e_{U_2}) + \bdel(\cl{V_1},e_{V_1})
			\end{split}\tag{by \eqref{eqn:case1del1}}\\
			&\leq -\big(\bdel(\cl{U_1}, e_{U_1})  + \bdel(\cl{V_2},e_{V_{2}})
			+ \del(\cl{U_2},e_{U_2}) + \del(\cl{V_1},e_{V_1})\big) \qquad\qquad \tag{by inductive hypothesis} \\
			&\quad + \del(Z_1+u'_1z_1,u'_1z_1) + \del(\cl Y, e_Y) + \del(Z_{2}+v'_{2}z_{2},v'_{2}z_{2}) \\
			&= -\big(\bdel(\cl{U_1}, e_{U_1})  + \bdel(\cl{V_2},e_{V_{2}})
			+ \del(\cl{U_2},e_{U_2}) + \del(\cl{V_1},e_{V_1})\big) \\
			&\quad
			-\big(\del(Z_2+u'_2z_2,u'_2z_2) + \del(\cl Y, e_Y)  + \del(Z_{1}+v'_{1}z_{1},v'_{1}z_{1})\big) \\
			&\quad
			+\big(\del(Z_2+u'_2z_2,u'_2z_2) + \del(\cl Y, e_Y)  + \del(Z_{1}+v'_{1}z_{1},v'_{1}z_{1})\big) \\
			&\quad + \del(Z_1+u'_1z_1,u'_1z_1) + \del(\cl Y, e_Y) + \del(Z_{2}+v'_{2}z_{2},v'_{2}z_{2})\\
			&\leq 1 + \del(Z_1+u'_1z_1,u'_1z_1) + \del(\cl Y, e_Y)  + \del(Z_{2}+v'_{2}z_{2},v'_{2}z_{2}) \qquad\qquad \tag{by \eqref{eqn:case1del2}} \\
			&\quad + \del(Z_2+u'_2z_2,u'_2z_2) + \del(\cl Y, e_Y)  + \del(Z_{1}+v'_{1}z_{1},v'_{1}z_{1}).
		\end{align*}
		This gives
		\begin{align*}
			-2&\leq\del(Z_1+u'_1z_1,u'_1z_1)  + \del(Z_{2}+v'_{2}z_{2},v'_{2}z_{2}) + \del(Z_2+u'_2z_2,u'_2z_2) + \del(Z_{1}+v'_{1}z_{1},v'_{1}z_{1})
			+ 2\del(\cl Y, e_Y)  \\
			&\leq -2,
		\end{align*}
		since by inductive hypothesis, all the terms are each at most $-\frac{1}{2}$ except $\del(\cl Y, e_Y)=0$ when $Y$ is a trivial chain.
		Hence, $\del(G,e)=-1$,
		\begin{align*}
		    \del(Z_1+u'_1z_1,u'_1z_1)   =\del(Z_{2}+v'_{2}z_{2},v'_{2}z_{2})  =\del(Z_2+u'_2z_2,u'_2z_2)  =\del(Z_{1}+v'_{1}z_{1},v'_{1}z_{1})
			=-\frac{1}{2},
		\end{align*}
		and $\del(\cl Y,e_Y)=0$ (i.e., $Y$ is a trivial chain).
		By Lemma \ref{lem:block3vx}, $Z_1$ and $Z_2$ are single vertices. So for $i\in [2]$, $\del(Z_i+u'_iz_i,u'_iz_i)= \del(Z_i+v'_iz_i,v'_iz_i)= -\frac{1}{2}$.
		Hence, from \eqref{eqn:case1del1} and \eqref{eqn:case1del2}, we have
		\begin{align*}
		    0&\leq \del(\cl{U_1},e_{U_1})+\del(\cl{V_2},e_{V_2})+\bdel(\cl{U_2},e_{U_2})+\bdel(\cl{V_1},e_{V_1}) \text{ and}\\
		    0&\leq \del(\cl{U_2},e_{U_2})+\del(\cl{V_1},e_{V_1})+\bdel(\cl{U_1},e_{U_1})+\bdel(\cl{V_2},e_{V_2}).
		\end{align*}
		Adding these two inequalities, it follows from the inductive hypothesis that both inequalities hold with equality, so $U_i,V_j$ are all tight, and
		$$\del(\cl{U_1},e_{U_1})-\del(\cl{U_2},e_{U_2})=\del(\cl{V_1},e_{V_1})-\del(\cl{V_2},e_{V_2}).$$
		Moreover, for $i\in[2]$, by considering an even cover in $\ec(G,e)$ through $U_i$ and $V_i$, we obtain
		\begin{align*}
		    -1=\del(G,e)&\leq \del(\cl{U_i},e_{U_i})+\del(\cl{V_i},e_{V_i})+\bdel(\cl{U_{3-i}},e_{U_{3-i}})+\bdel(\cl{V_{3-i}},e_{V_{3-i}}) \\
		    &=2\big(\del(\cl{U_i},e_{U_i})-\del(\cl{U_{3-i}},e_{U_{3-i}})\big),
		\end{align*}
		so $\del(\cl{U_1},e_{U_1})-\del(\cl{U_2},e_{U_2})\in\{-\frac{1}{2},0,\frac{1}{2}\}$ by half-integrality.
		Since $Z_i=\{z_i\}$ for $i\in[2]$ and $Y$ is a trivial chain, $(G,e)$ is a rooted $K_4$-chain with chains $U_1,U_2,V_1,V_2$, which is tight and balanced by the above.
		This proves that $(G,e)$ satisfies (a) of  \ref{main1}.

		\begin{figure}[H]
		    \centering
		    \includegraphics[scale=0.3]{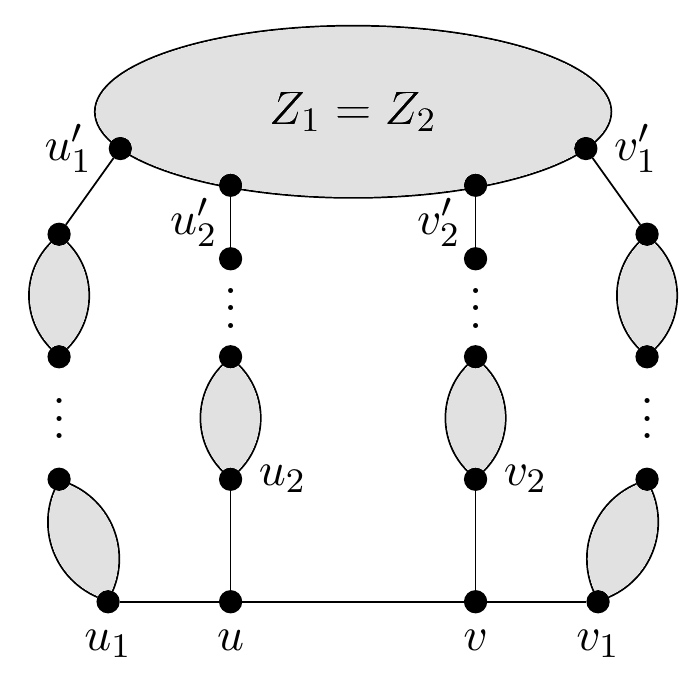}
		    \caption{$Z_1 = Z_2$}
		    \label{fig:Z1=Z2}
		\end{figure}

		{\bf Case 2: $Z_1=Z_2$}.

		Let $Z:=Z_1=Z_2$.
		Then $u_1',u_2',v_1',v_2'$ are distinct vertices (since $G$ is subcubic and $Z$ is 2-connected),
		and $G$ is the union of $U_1,U_2,V_1,V_2,Z$, and the edge $e$. Note that for all $i,j\in[2]$,
		\begin{align*}
			n(G) &= n(\cl{U_1}) + n(\cl{U_2}) + n(\cl{V_1}) + n(\cl{V_2}) + n(Z+u'_iv'_j) + 2 \\
			n_2(G) &= n_2(\cl{U_1}) + n_2(\cl{U_2}) + n_2(\cl{V_1}) + n_2(\cl{V_2}) + n_2(Z+u'_iv'_j) - 2.
		\end{align*}

		For $i,j\in [2]$, let $F\in \ec(G,e)$ be an even cover through $U_i$ and $V_j$.
		This corresponds to a tuple $(F_{U_1},F_{U_2},F_{V_1},F_{V_2},F_{Z})$ where
		\begin{itemize}
		    \item $F_{U_i}\in \ec(\cl{U_i},e_{U_i}),\ F_{V_j}\in \ec(\cl{V_j},e_{V_j}),\ F_{Z}\in \ec(Z+u_i'v_j',u_i'v_j')$, and
		    \item $F_{U_{3-i}}\in \bec(\cl{U_{3-i}},e_{U_{3-i}}),\ F_{V_{3-j}}\in \bec(\cl{V_{3-j}},e_{V_{3-j}})$,
		\end{itemize}
		which gives
		\begin{align*}
			\minexc(G,e)
			&\leq \minexc(\cl{U_i},e_{U_i}) + \minexc(\cl{V_j},e_{V_j}) + \minexc(Z+u'_iv'_j,u'_iv'_j) +\bminexc(\cl{U_{3-i}},e_{U_{3-i}})+\bminexc(\cl{V_{3-j}},e_{V_{3-j}}) \\
			&= \frac{n(G)+n_2(G)}{4} + \del(\cl{U_i},e_{U_i}) + \del(\cl{V_j},e_{V_j}) + \del(Z+u'_iv'_j,u'_iv'_j) +\bdel(\cl{U_{3-i}},e_{U_{3-i}})+\bdel(\cl{V_{3-j}},e_{V_{3-j}}).
		\end{align*}
	   Hence, for all $i,j \in [2]$,
		\begin{align}
		\label{eqn:case2dels}
		\begin{split}
			\del(G,e)  &\leq \del(\cl {U_i},e_{U_i}) + \del(\cl {V_j},e_{V_j}) + \del(Z+u'_iv'_j,u'_iv'_j) +\bdel(\cl {U_{3-i}},e_{U_{3-i}})+\bdel(\cl {V_{3-j}},e_{V_{3-j}})
		\end{split}
		\end{align}

		We now show that $\del(G,e)\leq -\frac{3}{2}$, which completes the proof of Theorem \ref{thm:main}.
		Suppose to the contrary that $\del(G,e) \geq -1$. Then by \eqref{eqn:case2dels} and the inductive hypothesis,
		\begin{align*}
			-1 &\leq \del(\cl{U_i},e_{U_i}) + \del(\cl{V_j},e_{V_j})  +\bdel(\cl{U_{3-i}},e_{U_{3-i}})+\bdel(\cl{V_{3-j}},e_{V_{3-j}})+ \del(Z+u'_iv'_j,u'_iv'_j) \tag{by \eqref{eqn:case2dels}}\\
			&\leq - \big(\bdel(\cl{U_i},e_{U_i}) + \bdel(\cl{V_j},e_{V_j})  +\del(\cl{U_{3-i}},e_{U_{3-i}})+\del(\cl{V_{3-j}},e_{V_{3-j}})\big) \qquad\quad \tag{by \ref{maindelbdel}} \\
			&\quad + \del(Z+u'_iv'_j,u'_iv'_j) \\
			&= - \big(\bdel(\cl{U_i},e_{U_i}) + \bdel(\cl{V_j},e_{V_j}) + \del(Z+u'_{3-i}v'_{3-j},u'_{3-i}v'_{3-j}) +\del(\cl{U_{3-i}},e_{U_{3-i}})+\del(\cl{V_{3-j}},e_{V_{3-j}})\big) \\
			&\qquad + \del(Z+u'_iv'_j,u'_iv'_j)+\del(Z+u'_{3-i}v'_{3-j},u'_{3-i}v'_{3-j}) \\
			&\leq 1 + \del(Z+u'_iv'_j,u'_iv'_j)+\del(Z+u'_{3-i}v'_{3-j},u'_{3-i}v'_{3-j}). \tag{by \eqref{eqn:case2dels}}
		\end{align*}
		Hence for $i,j\in [2]$,
		\begin{equation} \label{eqn:u1v1+u2v2}
		-2\leq \del(Z+u'_iv'_j,u'_iv'_j)+\del(Z+u'_{3-i}v'_{3-j},u'_{3-i}v'_{3-j}).
		    \end{equation}
		On the other hand, applying Lemma \ref{lem:block3vx} to $u'_i,v'_1,v'_2$ and $v'_j,u'_1,u'_2$, we have for all $i,j\in[2]$
		\begin{align} \label{eqn:u1v1+u1v2}
		    \begin{split}
		    \del(Z+u_i'v_1',u_i'v_1')+\del(Z+u_i'v_2',u_i'v_2') &\leq -2 \text{ and}\\
		    \del(Z+u_1'v_j',u_1'v_j')+\del(Z+u_2'v_j',u_2'v_j') &\leq -2.
		    \end{split}
		\end{align}
		Now, setting $i=j=1$ and setting $i=1$ and $j=2$ in \eqref{eqn:u1v1+u2v2}, we have
		\begin{align*}
			-4 \leq  \del(Z+u'_1v'_1,u'_1v'_1)+\del(Z+u'_2v'_2,u'_2v'_2)+\del(Z+u'_1v'_2,u'_1v'_2)+\del(Z+u'_2v'_1,u'_2v'_1).
		\end{align*}
        On the other hand, setting $i=1$ and $i=2$ in the first inequality of \eqref{eqn:u1v1+u1v2}, we have
		\begin{align*}
			\del(Z+u'_1v'_1,u'_1v'_1)+\del(Z+u'_1v'_2,u'_1v'_2)+\del(Z+u'_2v'_1,u'_2v'_1)+\del(Z+u'_2v'_2,u'_2v'_2) \leq -4.
		\end{align*}
		We thus have equality everywhere.
		In particular, $\del(G,e)=-1$ and we have equality in \eqref{eqn:u1v1+u2v2} and \eqref{eqn:u1v1+u1v2}, which implies that for all $i,j\in[2]$,
		\begin{equation} \label{eqn:deluivj}
		    \del(Z+u_i'v_j',u_i'v_j') = -1.
		\end{equation}

		Since $Z$ is 2-connected, $u_i'v_j'$ is not contained in any 2-edge-cut in $Z+u_i'v_j'$.

		We claim that no $(Z+u_i'v_j',u_i'v_j')$ satisfies (a) of \ref{main1}.
		Indeed,	suppose without loss of generality that $(Z+u_1'v_1',u_1'v_1')$ satisfies (a) of \ref{main1}.
		Then $(Z+u_1'v_1',u_1'v_1')$ is a balanced tight rooted $K_4$-chain.
		Let $X_1,X_2,Y_1,Y_2$ denote its chains, where $X_i$ has endpoints $u_1',z_i$ and $Y_i$ has endpoints $v_1',z_i$ for $i\in[2]$, and $z_1z_2\in E(Z)$. Then
		\begin{align*}
		    \del(\cl{X_1},e_{X_1})-\del(\cl{X_2},e_{X_2})=\del(\cl{Y_1},e_{Y_1})-\del(\cl{Y_2},e_{Y_2}).
		\end{align*}
		Since $u_2'$ has degree 2 in $Z+u_1'v_1'$ and $z_1,z_2$ have degree 3 in $Z+u_1'v_1'$, $u_2'$ is in the interior of one of these four chains; by the symmetry between $z_1$ and $z_2$, we may assume that $u_2'$ is in $X_1$ or $Y_1$.
		Let $C\in\{X_1,Y_1\}$ denote the chain containing $u_2'$, and write $C=ce_0B_1e_1B_2\dots B_ke_kc'$ with chain-blocks $(\cl{B_i},\cl{e_i})$, where $c=u_1'$ if $C=X_1$ and $c=z_1$ if $C=Y_1$.
		Since $C$ is tight, we have $\bdel(\cl{B_i},\cl{e_i})=-\del(\cl{B_i},\cl{e_i})$ for all $i\in[k]$ by Proposition \ref{prop:tight} and the inductive hypothesis.
		Let $\ell\in[k]$ be the unique index such that $u_2'\in V(B_\ell)$, and let $s$ denote the vertex of $B_\ell$ incident with $e_{\ell-1}$.
		Note that $n(B_\ell+su_2')=n(\cl{B_\ell})$ and $n_2(B_\ell+su_2')=n_2(\cl{B_\ell})$, and that
		\begin{align*}
		    n(Z+u_2'v_1') &= n(\cl{X_1})+n(\cl{X_2})+n(\cl{Y_1})+n(\cl{Y_2})+4, \\
		    n_2(Z+u_2'v_1') &= n_2(\cl{X_1})+n_2(\cl{X_2})+n_2(\cl{Y_1})+n_2(\cl{Y_2}).
		\end{align*}

		First suppose $C=X_1$, and let $F\in \ec(Z+u_2'v_1',u_2'v_1')$ be an even cover through $Y_1$, $z_1z_2$, $X_2$, and $B_1,\dots,B_\ell$.
		This corresponds to a tuple $(F_{Y_1},F_{X_2},F_{B_1},\dots,F_{B_k},F_{Y_2})$ where
		\begin{itemize}
		    \item $F_{Y_1}\in \ec(\cl{Y_1},e_{Y_1}),\ F_{X_2}\in \ec(\cl{X_2},e_{X_2}),\ F_{B_i}\in \ec(\cl{B_i},\cl{e_i})$ for $i\in[\ell-1]$, $F_{B_\ell}\in \ec(B_\ell+su_2',su_2')$, and
		    \item $F_{B_j}\in \bec(\cl{B_j},\cl{e_j})$ for $j=\ell+1,\dots,k$, $F_{Y_2}\in \bec(\cl{Y_2},e_{Y_2})$,
		\end{itemize}
		which gives
		\begin{align*}
			\minexc(Z+u_2'v_1',u_2'v_1')
			&\leq \minexc(\cl{Y_1},e_{Y_1}) + \minexc(\cl{X_2},e_{X_2}) + \sum_{i=1}^{\ell-1}\minexc(\cl{B_i},\cl{e_i}) + \minexc(B_\ell+su_2',su_2') \\
			&\quad + \sum_{j=\ell+1}^k\bminexc(\cl{B_j},\cl{e_j}) + \bminexc(\cl{Y_2},e_{Y_2}) \\
			&= \frac{n(Z+u_2'v_1')+n_2(Z+u_2'v_1')}{4} - 1 + \del(\cl{Y_1},e_{Y_1}) + \del(\cl{X_2},e_{X_2}) + \sum_{i=1}^{\ell-1}\del(\cl{B_i},\cl{e_i}) \\
			&\qquad\qquad + \del(B_\ell+su_2',su_2') + \sum_{j=\ell+1}^k\bdel(\cl{B_j},\cl{e_j}) + \bdel(\cl{Y_2},e_{Y_2}).
		\end{align*}
		Hence,
		\begin{align*}
		    \del(Z+u_2'v_1',u_2'v_1')
		    &\leq -1+\del(\cl{Y_1},e_{Y_1})+\del(\cl{X_2},e_{X_2})+\sum_{i=1}^{\ell-1}\del(\cl{B_i},\cl{e_i})+\del(B_\ell+su_2',su_2') \\
		    &\qquad +\sum_{j=\ell+1}^k\bdel(\cl{B_j},\cl{e_j})+\bdel(\cl{Y_2},e_{Y_2}) \\
		    &= -1+\del(\cl{X_1},e_{X_1})+\sum_{i=1}^{\ell-1}\del(\cl{B_i},\cl{e_i})+\del(B_\ell+su_2',su_2')-\sum_{j=\ell+1}^k\del(\cl{B_j},\cl{e_j}) \\
		    &= -1+2\sum_{i=1}^{\ell-1}\del(\cl{B_i},\cl{e_i})+\del(\cl{B_\ell},\cl{e_\ell})+\del(B_\ell+su_2',su_2') \tag{by Proposition \ref{prop:subcubicchains}}\\
		    &\leq -2,
		\end{align*}
		where the first equality holds since $Y_2$ and the chain-blocks of $X_1$ are tight and $\del(\cl{Y_1},e_{Y_1})-\del(\cl{Y_2},e_{Y_2})+\del(\cl{X_2},e_{X_2})=\del(\cl{X_1},e_{X_1})$, and the last inequality holds by the inductive hypothesis applied to $(\cl{B_i},\cl{e_i})$ for $i\in[\ell]$ and $(B_\ell+su_2',su_2')$.

		Now suppose $C=Y_1$. Similarly, by considering an even cover in $\ec(Z+u_2'v_1',u_2'v_1')$ through $Y_2$, $X_2$, $X_1$, and $B_1,\dots,B_\ell$, we obtain
		\begin{align*}
		    \del(Z+u_2'v_1',u_2'v_1')
		    &\leq -1+\del(\cl{Y_2},e_{Y_2})+\del(\cl{X_2},e_{X_2})+\del(\cl{X_1},e_{X_1})+\sum_{i=1}^{\ell-1}\del(\cl{B_i},\cl{e_i}) \\
		    &\qquad +\del(B_\ell+su_2',su_2')+\sum_{j=\ell+1}^k\bdel(\cl{B_j},\cl{e_j}) \displaybreak[3]\\
		    &= -1+2\del(\cl{X_2},e_{X_2})+\del(\cl{Y_1},e_{Y_1})+\sum_{i=1}^{\ell-1}\del(\cl{B_i},\cl{e_i}) \\
		    &\qquad +\del(B_\ell+su_2',su_2')-\sum_{j=\ell+1}^k\del(\cl{B_j},\cl{e_j}) \\
		    &= -1+2\del(\cl{X_2},e_{X_2})+2\sum_{i=1}^{\ell-1}\del(\cl{B_i},\cl{e_i}) \\
		    &\qquad +\del(\cl{B_\ell},\cl{e_\ell})+\del(B_\ell+su_2',su_2') \tag{by Proposition \ref{prop:subcubicchains}}\\
		    &\leq -2.
		\end{align*}
		In both cases, we have $\del(Z+u_2'v_1',u_2'v_1')\leq -2$, contradicting \eqref{eqn:deluivj}.

		So each $(Z+u_i'v_j',u_i'v_j')$ satisfies (b) or (d) of \ref{main1}.

		We claim that $u_i'v_j'\notin E(Z)$ for all $i,j\in [2]$ (hence $(Z+u_i'v_j',u_i'v_j')$ satisfies (d) of \ref{main1}). For,
		suppose without loss of generality that $u_1'v_1' \in E(Z)$.
		By the inductive hypothesis, (b) of \ref{main1} holds for $(Z+u_1'v_1',u_1'v_1')$, so suppressing $\{u_1',v_1'\}$ in $Z$ to an edge $e'$ results in a graph $Z'$ such that $(Z',e')$ is a near-minimal rooted $\theta$-chain.
		Let $C_1,C_2$ denote the two chains of $(Z',e')$. Assume without loss of generality that $v_2'\in V(C_1)$.
		Since $v_2'$ has degree 2 in $Z$, it is in the interior of $C_1$, and this implies that $Z-\{u_1',v_2'\}$ is connected and $v_2'u_1'\notin E(Z)$. Then $(Z+u_1'v_2',u_1'v_2')$ satisfies (d) of \ref{main1}, which implies that $Z - \{u_1',v_2'\}$ is disconnected, a contradiction.

		It follows that $(Z+u_i'v_j',u_i'v_j')$ satisfy (d) of \ref{main1} for all $i,j\in[2]$, so $(Z+u_i'v_j',u_i'v_j')$ is a rooted $\theta$-chain for all $i,j\in[2]$.
		Consider the rooted $\theta$-chain $(Z+u_1'v_1',u_1'v_1')$.
		Since $(Z+u_1'v_2',u_1'v_2')$ (respectively, $(Z+u_2'v_1',u_2'v_1')$) is a rooted $\theta$-chain, $\{v_2'\}$ (respectively, $\{u_2'\}$) is a block in one of the chains of $(Z+u_1'v_1',u_1'v_1')$.
		Let $C_1$ denote the subcubic chain of $Z$ with end points $\{u_1', v_1'\}$ not containing $v_2'$, and let $C_2$ denote the subcubic chain of $Z$ with end points  $\{u_1',v_2'\}$ not containing $v_1'$.
		Let $D$ denote the subcubic chain of $Z$ with end points $\{v_1', v_2'\}$ not containing $u_1'$.

		Then for $j\in [2]$, $n(Z+u_1'v_j') = n(\cl{C_1})+n(\cl{C_2})+n(\cl{D})+3$ and $n_2(Z+u_1'v_j')=n_2(\cl{C_1})+n_2(\cl{C_2})+n_2(\cl{D})+1$. Thus for each $j\in[2]$, by forming an even cover in $\ec(Z+u_1'v_j',u_1'v_j')$  using even covers from $\bec(\cl{C_j},e_{C_j})$, $\ec(\cl{D},e_D)$, and $\ec(\cl{C_{3-j}},e_{C_{3-j}})$, we obtain
		\begin{align*}
		    \del(Z+u_1'v_j',u_1'v_j') \leq -1 + \bdel(\cl{C_j},e_{C_j})+\del(\cl{D},e_D)+\del(\cl{C_{3-j}},e_{C_{3-j}}).
		\end{align*}
		Adding these two inequalities and using \eqref{eqn:deluivj}, we have
		\begin{align*}
		    0&\leq \del(\cl{C_1},e_{C_1})+\bdel(\cl{C_1},e_{C_1})+2\del(\cl D,e_D)+\del(\cl{C_2},e_{C_2})+\bdel(\cl{C_2},e_{C_2}) \\
		    &\leq 2\del(\cl D,e_D)
		\end{align*}
		by \ref{maindelbdel} applied to $(\cl{C_i},e_{C_i})$.
		It follows that $D$ is a trivial chain, and $v_1'v_2'\in E(Z)$.

		By symmetry, $u_1'u_2'\in E(Z)$. Thus, $\{u_1'u_2',v_1'v_2'\}$ is a 2-edge-cut in $Z$. Let $D_1,D_2$ denote the connected components of $Z-\{u_1'u_2',v_1'v_2'\}$ and (by relabeling $u_1',u_2'$ if necessary) assume $u_i',v_i'\in V(D_i)$ for $i\in [2]$.
		Then for $i,j\in [2]$, $n(Z+u_i'v_j') = n(\cl{D_1})+n(\cl{D_2})+4$ and $n_2(Z+u_i'v_j') = n_2(\cl{D_1})+n_2(\cl{D_2})+2$.
		Thus, for $k\in [2]$, by forming an even cover in $\ec(Z+u_{3-k}'v_{3-k}',u_{3-k}'v_{3-k}')$ using even covers from $\ec(\cl{D_k},e_{D_k})$ and $\bec(\cl{D_{3-k}},e_{D_{3-k}})$, we get
		\begin{align*}
		    \del(Z+u_{3-k}'v_{3-k}',u_{3-k}'v_{3-k}')
		    &\leq -\frac{3}{2} + \del(\cl{D_k},e_{D_k})+\bdel(\cl{D_{3-k}},e_{D_{3-k}}).
		\end{align*}
		Adding these two inequalities  and using \eqref{eqn:deluivj} and \ref{maindelbdel}, we have
		\begin{align*}
		    1 &\leq \del(\cl{D_1},e_{D_1})+\bdel(\cl{D_1},e_{D_1})+\del(\cl{D_2},e_{D_2})+\bdel(\cl{D_2},e_{D_2}) \leq 0,
		\end{align*}
		a contradiction. This completes the proof of Theorem \ref{thm:main}.

    \section{Extremal Examples}\label{sec:tight}
    In this section, we give a structural characterization of the extremal examples of Theorem \ref{thm:tspwalk}.
    Recall that for a simple 2-connected subcubic graph $G$ and any edge $e\in E(G)$, we have
    \begin{align*}
        \exc(G) &= \min\{\minexc(G,e)+2,\ \bminexc(G,e)\} \\
        &=\frac{n(G)+n_2(G)}{4} + \min\{\del(G,e)+2,\  \bdel(G,e)\}.
    \end{align*}
    So if either $\del(G,e)\leq -\frac{3}{2}$ or $\bdel(G,e)\leq \frac{1}{2}$ for any edge $e \in E(G)$, then $\exc(G) \leq \frac{n(G)+n_2(G)}{4}+\frac{1}{2}$.
    It follows that $\exc(G) = \frac{n(G) + n_2(G)}{4} + 1$ (equivalently, $\tsp(G) = \frac{5n(G)+ n_2(G)}{4} - 1$)  if and only if $(\del(G,e),\bdel(G,e)) = (-1,1)$ for all $e\in E(G)$.

    \begin{prop}
    \label{prop:del1tight}
    Let $G$ be a simple 2-connected subcubic graph and let $e$ be an edge of $G$.
    Then $(\del(G,e),\bdel(G,e))=(-1,1)$ if and only if either $G\cong K_4$ or $G$ is a minimal $\theta$-chain.
    \end{prop}
    \begin{proof}
    Suppose $(\del(G,e),\bdel(G,e))=(-1,1)$.
    Since $\del(G,e)=-1$, one of the four outcomes of \ref{main1} holds.
    If $G\cong K_4$ then we are done. Since $G$ is simple, (b) of \ref{main1} cannot occur. Moreover, (d) of \ref{main1} does not hold; otherwise, $(G,e)$ is a simple rooted $\theta$-chain and, by Lemma \ref{lem:baltighttheta} (ii), $\bdel(G,e) \le \frac{3}{2} + \del(\cl{C_1},e_{C_1}) + \del(\cl{C_2},e_{C_2}) \le 1/2$, a contradiction.
    Similarly, (a) of \ref{main1} does not hold unless $G\cong K_4$; otherwise, at least one of the chains $U_1,U_2,V_1,V_2$ of the rooted $K_4$-chain $(G,e)$ is nontrivial, and by forming an even cover in $\bec(G,e)$ using even covers from $\ec(\cl{U_i},e_{U_i})$ and $\ec(\cl{V_i},e_{V_i})$ for $i\in[2]$, we obtain $\bdel(G,e)\le 1+\sum_{i\in[2]}\left(\del(\cl{U_i},e_{U_i})+\del(\cl{V_i},e_{V_i})\right)\le 1/2$ by \ref{main1/2}, a contradiction.

     Thus (c) of \ref{main1} holds: there exists $e'\in E(G)$ such that $\{e,e'\}$ is a 2-edge-cut in $G$ and suppressing either subcubic chain $C$ of $G$ with end edges $e,e'$ yields a loop or a balanced tight rooted $\theta$-chain $(G/C,e_{G/C})$. Let $C$ be a subcubic chain of $G$ with end edges $e, e'$ such that $G/C$ is not a loop (which exists since $G$ is simple).
     Then by Proposition \ref{prop:contractchain} and \ref{maindelbdel},
     \begin{align*}
     -1&=\del(G,e) = \del(G/C,e_{G/C})+\del(\cl C, e_C) \leq -\big(\bdel(G/C,e_{G/C})+\bdel(\cl C, e_C)\big) = -\bdel(G,e) = -1. \end{align*}
     This implies that $(\del(G/C,e_{G/C}),\bdel(G/C,e_{G/C})) = (\del(\cl C,e_C),\bdel(\cl C,e_C)) = (-\frac{1}{2},\frac{1}{2})$, and thus $C$ is minimal and $(G/C,e_{G/C})$ is a minimal rooted $\theta$-chain (by Lemma \ref{lem:baltighttheta} (iii)). Therefore, by definition, $G$ is a minimal $\theta$-chain (since it is the internally disjoint union of $C$ and the two chains of $(G/C,e_{G/C})$, all of which are minimal).
    \end{proof}

    To give an alternate structural characterization of minimal (rooted) $\theta$-chains, we now describe an operation introduced in \cite{DKM17}. Let $H$ be a graph and $v\in V(H)$ be a vertex of degree 2. A {\it $\diamond$-operation on $H$ at $v$} deletes $v$ from $H$, adds a 4-cycle $D$ disjoint from $H-v$, and adds a matching between the neighbors of $v$ and two nonadjacent vertices in $D$.  See Figure \ref{fig:diamond}. We say that a graph is {\it $H$-constructible} if it can be obtained from $H$ by repeated $\diamond$-operations.
    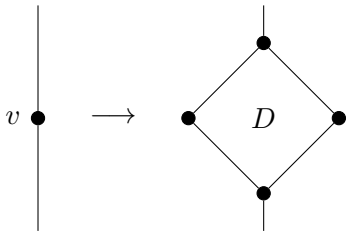
\begin{figure}[H]
        \centering
    \begin{tikzpicture}
	\node[style=vertex] (x) at (0,0) [label=left:$v$]{};
	\node[style=vertex] (a) at (2,0){};
	\node[style=vertex] (b) at (3,1){};
	\node[style=vertex] (c) at (4,0){};
	\node[style=vertex] (d) at (3,-1){};

	\node [fill=none] (arrow) at (1,0) {$\longrightarrow$};
    \node [fill=none] (D) at (3,0) {$D$};

	\draw[style=edge] (0,1.5) to (x) to (0,-1.5);
	\draw[style=edge] (b) to (3,1.5);
	\draw[style=edge] (d) to (3,-1.5);
	\draw[style=edge] (a) to (b) to (c) to (d) to (a);
    \end{tikzpicture}
    \caption{The $\diamond$-operation}
        \label{fig:diamond}
    \end{figure}

    It is observed in \cite{DKM17} that after each $\diamond$-operation, the excess of the new graph increases by 1 and the new quantity $\frac{n(G)+n_2(G)}{4}$ also increases by 1. We will consider $K_{2,3}$-constructible graphs and $K_4^-$-constructible graphs, where $K_4^-$ is the graph obtained from the complete graph $K_4$ by removing an edge. Note that $\exc(K_{2,3}) = \frac{n(K_{2,3}) + n_2(K_{2,3})}{4} + 1$; thus,  if $G$ is $K_{2,3}$-constructible then $\exc(G) = \frac{n(G) + n_2(G)}{4} + 1$.

    \begin{prop}[Dvo\v r\'ak et al. \cite{DKM17}] \label{prop:k23}
    Let $G$ be a simple 2-connected subcubic graph. If $G\cong K_4$ or $G$ is $K_{2,3}$-constructible, then $(\del(G,e),\bdel(G,e)) = (-1,1)$.
    \end{prop}

    We show that the converse of Proposition \ref{prop:k23} is also true, thereby giving a structural characterization of the extremal graphs for Theorem \ref{thm:tspwalk}. First, we have an observation similar to Proposition \ref{prop:k23}. The {\it center} of $K_4^-$ is the  edge whose endpoints both have degree 3.

    \begin{prop}\label{prop:minimal_theta_characterization}
    Let $(G,e)$ be a simple minimal rooted $\theta$-chain. Then $G$ is $K_4^-$-constructible, with the edge $e$ corresponding to the center of $K_4^-$.
    \end{prop}
    \begin{proof}
    By \ref{main1/2} and Lemma \ref{lem:baltighttheta} (iii), $(\del(G,e),\bdel(G,e)) = (-\frac{1}{2},\frac{1}{2})$. Let $C_1$ and $C_2$ be the chains of $(G,e)$.
    By the definition of a minimal rooted $\theta$-chain, for each $i\in[2]$, we have $(\del(\cl{C_i},e_{C_i}),\bdel(\cl{C_i},e_{C_i})) = (-\frac{1}{2},\frac{1}{2})$, so $(\cl{C_i},e_{C_i})$ is either a loop or a minimal rooted $\theta$-chain by \ref{main1/2} and Lemma \ref{lem:baltighttheta}.
    If $(\cl{C_i},e_{C_i})$ is not a loop, then by induction, it is $K_4^-$-constructible with $e_{C_i}$ corresponding to the center of $K_4^-$.
    It follows that $(G,e)$ is $K_4^-$-constructible with $e$ corresponding to the center of $K_4^-$.
    \end{proof}

    \begin{prop} \label{prop:minimalthetaK23-constructible}
    Let $G$ be a simple minimal $\theta$-chain.
    Then $G$ is $K_{2,3}$-constructible.
    \end{prop}
    \begin{proof}
    By definition, there exists a choice of three chains  $C_1,C_2,C_3$ of $G$ with common endpoints such that $G$ is the internally disjoint union $C_1\cup C_2\cup C_3$, and we have $(\del(\cl{C_i},e_{C_i}),\bdel(\cl{C_i},e_{C_i}))=(-\frac{1}{2},\frac{1}{2})$ for each $i\in[3]$.
    If $G\cong K_{2,3}$, then we are done.
    So we may assume without loss of generality that $(\cl{C_1},e_{C_1})$ is not a loop.
    Then it is a minimal rooted $\theta$-chain by Lemma \ref{lem:baltighttheta}, and by Proposition \ref{prop:minimal_theta_characterization}, it is $K_4^-$-constructible with the edge $e_{C_1}$ corresponding to the center of $K_4^-$.
    On the other hand, $(G/C_1,e_{G/C_1})$ is by definition a minimal rooted $\theta$-chain, so it is also $K_4^-$-constructible by Proposition \ref{prop:minimal_theta_characterization}, with $e_{G/C_1}$ corresponding to the center of $K_4^-$.
    It follows that $G$ is $K_{2,3}$-constructible.
    \end{proof}

    We thus have the following characterization of the extremal examples of Theorem \ref{thm:tspwalk}.

	\begin{theorem} \label{thm:extremalcharacterization}
	Let $G$ be a simple 2-connected subcubic graph. Then $\exc(G)\leq \frac{n(G)+n_2(G)}{4}+1$, with equality if and only if either $G\cong K_4$ or $G$ is $K_{2,3}$-constructible.
    \end{theorem}
	\begin{proof}
	Let $e\in E(G)$.
	If $\del(G,e)\leq -\frac{3}{2}$ or $\bdel(G,e)\leq \frac{1}{2}$, then $\exc(G)\leq \frac{n(G)+n_2(G)}{4}+\frac{1}{2}$.
	Otherwise, we have $(\del(G,e),\bdel(G,e))=(-1,1)$, or equivalently, $\exc(G) = \frac{n(G)+n_2(G)}{4}+1$.
	Now if $G \cong K_4$ or $G$ is $K_{2,3}$-constructible, then $(\del(G,e),\bdel(G,e))=(-1,1)$ by Proposition \ref{prop:k23}.
	Conversely, if $(\del(G,e),\bdel(G,e))=(-1,1)$, then by Propositions \ref{prop:del1tight} and \ref{prop:minimalthetaK23-constructible}, either $G\cong K_4$ or $G$ is $K_{2,3}$-constructible.
	\end{proof}

   \section{Algorithm}\label{sec:algo}

	We now provide an algorithm for finding a TSP walk of length at most $\frac{5n(G)+n_2(G)}{4} - 1$ in any simple 2-connected subcubic graph $G$. This is achieved by following the proof of Theorem \ref{thm:main} to construct an even cover $F$ of $G$ with $\exc(F)\leq \frac{n(G)+n_2(G)}{4}  + 1$. As noted by Dvo\v r\'ak et al. \cite{DKM17}, modifying this even cover to our desired TSP walk takes linear time.

	In the proof of Theorem \ref{thm:main}, we often have a choice of routing a cycle through certain subcubic chains and not through others. For each such chain $C$, we ``save'' $\del(\cl C,e_C)$ by going through $C$ and incur a ``cost'' $\bdel(\cl{C},e_C)$ by not going through $C$. The key idea of Theorem \ref{thm:main} is that these costs and savings are (at worst) balanced, i.e. $\del (\cl C,e_C) + \bdel(\cl C,e_C) \le 0$.
	Of course, for a given subcubic graph $G$ and an edge $e$, we cannot efficiently compute $\del(G,e)$ and $\bdel(G,e)$ exactly (unless P=NP).
	Instead, we compute ``worst-case'' estimates
	$(\Del(G,e),\BDel(G,e))$
	such that $(\del(G,e),\bdel(G,e)) \leq (\Del(G,e),\BDel(G,e))$ (coordinate-wise).

	The natural approach would be to determine exactly when $\del(G,e)=-\frac12$ or $\bdel(G,e)=\frac32$ using our characterization of the extremal examples in Theorem \ref{thm:main}, and assign $(\Del(G,e),\BDel(G,e)) = (-\frac12,\frac12)$ or $(-\frac32,\frac32)$ respectively (and assign $(-1,1)$ in all other cases).
	To check whether $(G,e)$ is a minimal rooted $\theta$-chain (for example), we would need to first check that it is a rooted $\theta$-chain (which takes linear time) and then recursively check that each of its two chains are also minimal, taking quadratic time overall.
	This approach would result in a cubic algorithm to produce the desired even covers.

	It turns out that a much simpler linear-time estimate is sufficient, and yields a quadratic-time algorithm to find the desired even covers.
    Indeed, by Lemma \ref{lem:baltighttheta}, if $(G,e)$ is a rooted $\theta$-chain (regardless of whether it is tight or balanced), then we have $(\del(G,e),\bdel(G,e))\leq (-\frac12,\frac12)$.
    And by Lemma \ref{lem:minimaltheta2}, if $G-e$ is simple and 2-connected and $(G_u,f_u)$ is a rooted $\theta$-chain (where $G_u$ is obtained from $G-e$ by suppressing an endpoint $u$ to an edge $f_u$), then we have $(\del(G,e),\bdel(G,e))\leq(-\frac32,\frac32)$.

    We thus define an algorithm $\scan(G,e)$ to estimate $(\del(G,e),\bdel(G,e))$ as follows.
    If $G$ is a loop or $G-e$ is 2-connected, $\scan(G,e)$ will assign
    \begin{align*}
        (\Del(G,e),\BDel(G,e)) =
        \left\{
	\begin{array}{ll}
		(-\frac12,\frac12)  & \mbox{if $(G,e)$ is a loop or a rooted $\theta$-chain,}\\
		(-\frac32,\frac32) & \mbox{if $(G_u,f_u)$ is a rooted $\theta$-chain,} \\
		(-1,1) &\mbox{otherwise.}
	\end{array}
\right.
    \end{align*}
    If $G-e$ is not 2-connected (and it is not a loop), then $(G,e)$ can be written as the closure $(\cl C,e_C)$ of a subcubic chain $C=xe_0B_1e_1 \cdots e_{k-1}B_ke_ky$ such that $k\geq 2$ (if $k=1$, then $G-e = \cl C-e_C$ is 2-connected or an isolated vertex).
    In this case, our estimate on $(G,e)$ will be the sum of the estimates of the chain-blocks $(\cl{B_i},e_{B_i})$ of $C$:
    \begin{align*}
    (\Del(G,e),\BDel(G,e)) =\sum_{i=1}^k (\Del(\cl{B_i},e_{B_i}), \BDel(\cl{B_i},e_{B_i})).
    \end{align*}

	For the remainder of this section, given a 2-connected subcubic graph $G$ and an edge $e=uv\in E(G)$ such that $G-e$ is simple and 2-connected, we let $u_1,u_2$ denote the two neighbors of $u$ in $G-e$, and denote by $G_u$ the graph obtained by deleting $e$ and suppressing $u$ to an edge $f_u=u_1u_2$.
	Note that computing $G_u$ and $f_u$ takes constant time.
	To resolve ambiguities in the choice of the vertex $u$ in the edge $e=uv$ (in the case where $\BDel(G,e)=\frac32$), we fix a linear ordering $\le$ of the vertices throughout, and assume that $u\leq v$ (except on line 10 of Algorithm \ref{algo:ec_1}, which may exchange $u$ and $v$).

	\begin{prop}
	Let $G$ be a loop or a 2-edge-connected subcubic graph and let $e = uv \in E(G)$ such that $G - e$ is simple. Then $\del(G,e) \le \Del(G,e)$ and $\bdel(G,e) \le \BDel(G,e)$.
	\end{prop}

	\begin{proof}
	First suppose $G$ is a loop or $G - e$ is 2-connected.
	If $(G,e)$ is a loop or a rooted $\theta$-chain, then by Lemma \ref{lem:baltighttheta}, $(\del(G,e),\bdel(G,e))\leq (-\frac12,\frac12) = (\Del(G,e),\BDel(G,e))$.
	If $(G_u,f_u)$ is a rooted $\theta$-chain, then by Lemma \ref{lem:minimaltheta2}, $(\del(G,e),\bdel(G,e))\leq (-\frac32,\frac32) = (\Del(G,e),\BDel(G,e))$.
	Otherwise, by Theorem \ref{thm:main} and Lemma \ref{lem:extremalG-e}, we have $(\del(G,e),\bdel(G,e))\leq (-1,1) = (\Del(G,e),\BDel(G,e))$.

	Now suppose $G-e$ is not 2-connected.
	Then we can write $(G,e)$ as the closure $(\cl C,e_C)$ of a subcubic chain $C=xe_0B_1e_1 \cdots e_{k-1}B_ke_ky$ where $k\geq 2$.
	By Proposition \ref{prop:subcubicchains} and by induction, we have
	\begin{align*}
	    (\del(\cl C,e_C),\bdel(\cl C,e_C)) = \sum_{i=1}^k (\del(\cl{B_i},e_{B_i}), \bdel(\cl{B_i},e_{B_i})) \leq \sum_{i=1}^k (\Del(\cl{B_i},e_{B_i}), \BDel(\cl{B_i},e_{B_i})) = (\Del(G,e),\BDel(G,e)).
	\end{align*}
	\end{proof}

    Checking whether $(G,e)$ is a rooted $\theta$-chain is equivalent to checking whether $G-\{u,v\}$ is disconnected, which can be done in linear time.
    More generally, we can determine the block structure of graphs with a depth first search (DFS) in $O(n(G)+|E(G)|)$ time (e.g. \cite{cormen2009introduction}), which is $O(n(G))$ when $G$ is subcubic.

    \begin{algorithm}[H]
    \small
    \SetKwInOut{Input}{Input}
    \SetKwInOut{Output}{Output}

    \Input{A loop or a  2-edge-connected subcubic graph $G$ and $e = uv \in E(G)$ such that $G - e$ is simple}
    \Output{A half integral vector $(\Del(G,e),\BDel(G,e))$.}
    \uIf{$G - e$ has a cut-vertex}{
    Write $(G,e)$ as the closure $(\cl C,e_C)$ of a subcubic chain $C=xe_0B_1e_1 \cdots e_{k-1}B_ke_ky$\;
  \Return{$\sum_{i = 1}^k \scan(\cl{B_i},e_{B_i})$}\;}
    \uIf{$G - \{u,v\}$ is nonempty and disconnected or $G$ is a loop}{\Return{ $(-\frac{1}{2},\frac{1}{2})$}\;}
    \uElseIf{$e$ has no parallel edge and $G_u-\{u_1,u_2\}$ is disconnected}{\Return{$(-\frac{3}{2},\frac{3}{2})$}\;
     }
     \uElse{
     \Return{$(-1,1)$}\;}
     \caption{\label{algo:scan} $\scan(G,e)$}
    \end{algorithm}

	\begin{prop}\label{prop:scan}
	$\scan(G,e)$ can be computed in $O(n(G))$ time.
	\end{prop}

	\begin{proof}
	If $\scan(G,e)$ returns on lines 5, 7, or 9, then it performs at most three depth first searches, thus requiring $O(n(G))$ time. Now suppose $\scan(G,e)$ returns on line 3; that is, $(G,e)$ is the closure of a subcubic chain $C=xe_0B_1e_1 \cdots e_{k-1}B_ke_ky$ where $k\geq 2$.
	For all $i \in [k]$, $\cl{B_i}-e_{B_i}$ is either 2-connected or a single vertex, so $\scan(\cl{B_i},e_{B_i})$ will not execute line 2. Thus $\scan(G,e)$ requires a depth first search on an input of size $n(G)$ on line 1 and at most three depth first searches for each $\cl{B_i}$, $i \in [k]$. As $\sum_{i = 1}^{k} n(\cl{B_i}) = n(G)$, we have that in all cases, $\scan(G,e)$ requires $O(n(G))$ time.
	\end{proof}

	We will define two algorithms $\alg(G,e)$ and $\balg(G,e)$ which will return an even cover $F$ in $\ec(G,e)$ and $\bec(G,e)$ respectively such that $\exc(F)\leq \frac{n(G)+n_2(G)}{4}+\Del(G,e)+2$ and $\exc(F)\leq \frac{n(G)+n_2(G)}{4}+\BDel(G,e)$ respectively.
    For convenience, we wrap these two algorithms in a main algorithm $\Algo$ with preprocessing to handle the base case (where $(G,e)$ is a loop) and the case where $G-e$ is not 2-connected.

\begin{algorithm}[H]
 \small
\SetKwInOut{Input}{Input}
\SetKwInOut{Output}{Output}
  \Input{A loop or a 2-edge-connected subcubic graph $G$ and $e \in E(G)$ such that $G-e$ is simple, and a binary input flag}
  \Output{$F\in \ec(G,e)$ such that $\exc(F) \le \frac{n(G) + n_2(G)}{4} + \Del(G,e) + 2$ (if $\flag == \true$) or $F \in \bec(G,e)$ such that $\exc(F) \le \frac{n(G) + n_2(G)}{4} + \BDel(G,e)$ (if $\flag == \false)$}
  \uIf{$G$ is a loop}{\uIf{$\flag == \true$}{\Return{$F = G$}\;}
   \uElse{\Return{$F = G-e$}\;}
}
  \uIf{$G-e$ is not 2-connected}{
  Write $(G,e)$ as the closure $(\cl C,e_C)$ of a subcubic chain $C = xe_0B_1e_1B_2 \ldots e_{k-1}B_ke_ky$\;
  Let $F_i = \Algo(\cl{B_i},e_{B_i},\flag)$ for all $i \in [k]$\;
  \uIf{$\flag == true$}{\Return{$F = \bigcup_{i = 1}^{k}(F_i - e_{B_i}) + e + \{e_i:i\in[k-1]\}$}\;
  }
  \uElse{
  \Return{$F = \bigcup_{i = 1}^k F_i$}\;}
  }
  Let $(\Del,\BDel) = \scan(G,e)$\;
  \uIf{$\flag == \true$}{\Return{$F = \alg(G,e,\Del)$}\;}
  \uElse{  \Return{$F = \balg(G,e)$\;
  }
  }
  \caption{\label{algo:main} $\Algo(G,e,\flag)$}
\end{algorithm}
    For the remainder of the section, we let $\falg:\mbn\to\mbn$ denote a superadditive function (i.e. $\falg(n_1)+\falg(n_2)\leq \falg(n_1+n_2)$ for all $n_1,n_2\in\mbn$) such that $\Algo(G,e,\flag)$ takes at most $\falg(n)$ steps on inputs of size at most $n$.
    We will show in the end that we can take $\falg(n)=O(n^2)$.

    We now give the algorithm $\balg(G,e)$ used in line 17 of $\Algo(G,e,\flag)$, which produces an even cover $F\in \bec(G,e)$ with $\exc(F)\leq \frac{n(G)+n_2(G)}{4}+\BDel(G,e)$.
    Recall that $(G_u,f_u)$ is obtained from $G$ and $e=uv$ by deleting $e$ and suppressing $u$ to an edge $f_u=u_1u_2$.

      \begin{algorithm}[H]
    \small
    \SetKwInOut{Input}{Input}
    \SetKwInOut{Output}{Output}

  \Input{A subcubic graph $G$ and $e =uv\in E(G)$ such that $G - e$ is simple and 2-connected}
  \Output{An even cover $F\in \bec(G,e)$ with $\exc(F) \le \frac{n(G) + n_2(G)}{4} + \BDel(G,e)$ where $\BDel(G,e) = \scan(G,e)_2$}
    Let $F' = \Algo(G_u,f_u,\true)$\;
    \Return{$F = (F' - f_u) + \{u\} + \{u_1u,uu_2\}$}\;
      \caption{\label{algo:bec} $\balg(G,e)$}
    \end{algorithm}

    \begin{prop}\label{prop:algo_bec_3/2}
	Suppose $\Algo$ is correct on inputs of size less than $n$. Then $\balg$ is correct and takes $\falg(n-1)+O(1)$ time for all inputs of size less than or equal to $n$.
	\end{prop}

	\begin{proof}
	We clearly have $F \in \bec(G,e)$.
	We claim that $\Del(G_u,f_u)+2 \leq \BDel(G,e)$.
	If $\BDel(G,e)=\frac32$, there is nothing to prove (since $\Del\leq -\frac12$).
	If $\BDel(G,e)=1$, then $(G_u,f_u)$ is not a rooted $\theta$-chain, so $\Del(G_u,f_u)\leq -1$.
	Finally, suppose $\BDel(G,e) = \frac{1}{2}$. Then $(G,e)$ is a rooted $\theta$-chain. This implies that $(G_u,f_u)$ is the closure $(\cl C,e_C)$ of a subcubic chain $C$ with at least three blocks, so $\Del(G_u,f_u) = \Del(\cl C,e_C) \le -\frac{3}{2}$.
	It follows that
	\begin{align*}
	    \exc(F) = \exc(F') \le \frac{n(G) + n_2(G)}{4} + \Del(G_u,f_u) + 2 \le \frac{n(G) + n_2(G)}{4} + \BDel(G,e).
	\end{align*}

	For the time complexity, note that $\Algo$ is called only once on $(G_u,f_u)$, which takes $\falg(n(G_u))=\falg(n-1)$ time. The remaining lines require constant time, thus $\balg$ runs in $\falg(n-1)+O(1)$ time.
	\end{proof}

    We now give the algorithm $\alg(G,e,\Del)$ in line 15 of $\Algo$, which produces an even cover $F \in \ec(G,e)$ such that $\exc(F)\leq \frac{n(G)+n_2(G)}{4}+\Del(G,e)+2$. For clarity of presentation, we split the algorithm into three cases depending on the value $\Del$. We first describe the case $\Del = -\frac{1}{2}$.

    \begin{algorithm}[H]
	    \small
        \SetKwInOut{Input}{Input}
        \SetKwInOut{Output}{Output}

        \Input{A subcubic graph $G$ and $e=uv \in E(G)$ such that $G - e$ is simple and 2-connected, and $\Del(G,e)=-\frac{1}{2}$ (i.e. $(G,e)$ is a rooted $\theta$-chain)}
        \Output{An even cover $F\in \ec(G,e)$ with $\exc(F) \le \frac{n(G) + n_2(G)}{4} + \frac{3}{2}$}

        Determine the subcubic chains $C_1$ and $C_2$ of $(G,e)$ with a DFS\;
        Let $(\Del(C_1),\BDel(C_1)) = \scan(\cl{C_1},e_{C_1})$ and
        let $(\Del(C_2),\BDel(C_2)) = \scan(\cl{C_2},e_{C_2})$\;
        Relabel if necessary so that $\Del(C_1) + \BDel(C_2) \le 0$\;
        Let $F_1 = \Algo(\cl{C_1},e_{C_1},\true)$ and $F_2 = \Algo(\cl{C_2},e_{C_2},\false)$\;
        Let $v'$ be the neighbor of $v$ in $C_1$ and let $u'$ be the neighbor of $u$ in $C_1$\;
        \Return{$F = (F_1 - e_{C_1})\cup F_2 +\{u,v\} +\{u'u,uv,vv'\}$}\;
        \caption{\label{algo:ec_1/2} $\alg(G,e,-\frac{1}{2})$}
	\end{algorithm}

    \begin{prop}\label{prop:ec_1/2}
    Suppose $\Algo$ is correct on inputs of size less than $n=n(G)$. Then $\alg(G,e,-\frac12)$ is correct and takes $\falg(n-1)+O(n)$ time for all input graphs of size less than or equal to $n$.
    \end{prop}

    \begin{proof}
    For correctness, first note that the relabeling step on line 3 is always possible as $\Del(C_i) = -\BDel(C_i)$ for $i \in [2]$.
    Since $n(G) = n(\cl{C_1}) + n(\cl{C_2}) + 2$, $n_2(G) = n_2(\cl{C_1}) + n_2(\cl{C_2})$, and $\exc(F) = \exc(F_1)+\exc(F_2)$, we have
    \begin{align*}
        \exc(F) &= \exc(F_1) + \exc(F_2)\\
        &\le \frac{n(\cl{C_1})+n_2(\cl{C_1})}{4} + \Del(C_1) + 2  + \frac{n(\cl{C_2}) + n_2(\cl{C_2})}{4} + \BDel(C_2)\\
        &\le \frac{n(G) + n_2(G)}{4} + \frac{3}{2}.
    \end{align*}

    For the time complexity, line 1 requires $O(n)$ time. By Proposition \ref{prop:scan}, line 2 requires $O(n(\cl{C_1})) + O(n(\cl{C_2})) = O(n)$ time. By induction, line 4 takes $\falg(n(\cl{C_1})) + \falg(n(\cl{C_2})) \leq \falg(n-1)$ time. Thus, in total, $\alg(G,e,-\frac12)$ takes $\falg(n-1)+O(n)$ time of inputs of size $n$.
    \end{proof}

    Before we handle the analysis of $\alg(G,e,-1)$, we first give an important subroutine which is an algorithmic version of Lemma \ref{lem:block3vx}.

\begin{algorithm}[H]
\small
\SetKwInOut{Input}{Input}
\SetKwInOut{Output}{Output}
  \Input{A simple 2-connected subcubic graph $Z$ and distinct vertices $u,v_1,v_2$ of degree 2 in $Z$}
  \Output{$F \in \ec(Z + uv_i,uv_i)$ for some $i \in [2]$ with $\exc(F) \le \frac{n(Z+uv_i) + n_2(Z+uv_i)}{4} + 1$}
   For each $i\in[2]$, let $(\Del_i,\BDel_i) = \scan(Z+uv_i,uv_i)$\;
   \uIf{ $\Del_i \le -1$ for some $i \in [2]$}{\Return $F = \Algo(Z+uv_i,uv_i,\true)$\;}
   Let $C_{i,1},C_{i,2}$ denote the two subcubic chains of $(Z+uv_i,uv_i)$, $i\in [2]$\;
   Let $(\Del(C_{i,j}),\BDel(C_{i,j})) = \scan(\cl{C_{i,j}},e_{C_{i,j}})$ for $i,j \in [2]$\;
   Relabel if necessary so that $\Del(C_{1,1}) + \BDel(C_{1,2}) \le -\frac{1}{2}$\;
   Let $F_1 = \Algo(\cl{C_{1,1}},e_{C_{1,1}},\true)$ and $F_{2} = \Algo(\cl{C_{1,2}},e_{C_{1,2}},\false)$\;
  Let $u'$ be the neighbor of $u$ in $C_{1,1}$ and $v'$ be the neighbor of $v_1$ in $C_{1,1}$\;
  \Return{$F = (F_1 - e_{C_{1,1}}) \cup F_{2} + \{u,v_1\} + \{u'u,uv_1,v_1v'\}$\;}
      \caption{\label{algo:Z_subroutine} $\subroutine(Z,u,v_1,v_2)$}
\end{algorithm}

	\begin{prop}\label{prop:algo_Z_1=Z_2}
	Suppose $\Algo$ is correct for all inputs of size less than or equal to $n=n(Z)$. Then $\subroutine$ is correct and takes $\falg(n)+O(n)$ time for all inputs of size less than or equal to $n$.
	\end{prop}

	\begin{proof}
	We first analyze correctness. If we return on line 3, by correctness of $\Algo$, we have $\exc(F) \le \frac{n(Z+uv_i) + n_2(Z+uv_i)}{4} + 1$.
	So assume $\Del_i=\Del(Z+uv_i,uv_i) = -\frac{1}{2}$ for both $i \in [2]$. Thus both $(Z+uv_i,uv_i)$ are rooted $\theta$-chains, which implies that $v_{3-i}$ is a trivial block in one of the chains $C_{i,1}$ and $C_{i,2}$.
    This then implies that $\Del(C_{i,1}) \neq \Del(C_{i,2})$ for some $i\in[2]$.
    Thus the relabeling step on line 6 is always possible.

	Now consider the even cover $F$ returned on line 9. As $n(Z+uv_1) = n(\cl{C_{1,1}}) + n(\cl{C_{1,2}}) + 2$, $n_2(Z+uv_1) = n_2(\cl{C_{1,1}}) + n_2(\cl{C_{1,2}})$, and $\Del(C_{1,1}) + \BDel(C_{1,2}) \le -\frac{1}{2}$, we have
	\begin{align*}
        \exc(F) &= \exc(F_1) + \exc(F_2)\\
        &\leq \frac{n(\cl{C_{1,1}})+n_2(\cl{C_{1,1}})}{4} + \Del(C_{1,1}) + 2  + \frac{n(\cl{C_{1,2}}) + n_2(\cl{C_{1,2}})}{4} + \BDel(C_{1,2})\\
        &\leq \frac{n(Z+uv_1)+n_2(Z+uv_1)}{4}+\Del(C_{1,1}) + \BDel(C_{1,2}) + \frac32\\
        &\leq \frac{n(Z+uv_1) + n_2(Z+uv_1)}{4} + 1.
    \end{align*}

    For the time complexity, as $n(\cl{C_{1,1}}) + n(\cl{C_{1,2}}) < n$, lines 3 and 7 both take at most $\falg(n)$ time. Furthermore, by Proposition \ref{prop:scan}, the remaining lines require $O(n)$ time. Since we call exactly one of line 3 or 7, $\subroutine(Z,u,v_1,v_2)$ takes $\falg(n)+O(n)$ time.
	\end{proof}

	We are now ready to present $\alg(G,e,-1)$.
	If $e$ has a parallel edge, then we follow the first claim in the proof of Theorem \ref{thm:main}: letting $x$ and $y$ denote the neighbors of $u$ and $v$, respectively, not in $\{u,v\}$, and letting $G'$ denote the graph obtained from $G$ by suppressing $\{u,v\}$ to an edge $e'=xy$ (a loop if $x=y$), $\alg(G,e,-1)$ lets $F'=\Algo(G',e',\true)$ and returns $F=(F'-e')+\{u,v\}+\{xu,e,vy\}$.
	Otherwise, $G$ is simple and $\alg(G,e,-1)$ is as follows.

    \begin{algorithm}[H]
	    \small
        \SetKwInOut{Input}{Input}
        \SetKwInOut{Output}{Output}

        \Input{A simple subcubic graph $G$ and $e=uv \in E(G)$ such that $G - e$ is 2-connected, and $\Del(G,e)=-1$.}

        \Output{$F\in \ec(G,e)$ with $\exc(F) \le \frac{n(G) + n_2(G)}{4} + 1$}
         Let $Z_1$ and $Z_2$ be the blocks (or single vertices) of $G-\{u,v\}$ as defined in Claim \ref{clm:Z_1Z_2}\;
         Define vertices $u_i,u_i',v_j,v_j'$ and subcubic chains $U_i,V_j$ for $i,j\in[2]$, as in the proof of Theorem \ref{thm:main}\;
        For each $C\in\{U_1,U_2,V_1,V_2\}$, let $(\Del(C),\BDel(C)) = \scan(\cl{C},e_{C})$ if $C$ is nontrivial, and let $(\Del(C),\BDel(C))=(0,0)$ and $F_C=\emptyset$ if $C$ is trivial\;
         \uIf{$Z_1 \neq Z_2$}{
            Relabel vertices as necessary so that $\Del(U_1)+\Del(V_2) + \BDel(U_2)+\BDel(V_1) \leq 0$\;
            Let $Z = Z_1 \cup Z_2 \cup Y$, where $Y$ is the subcubic chain from $Z_1$ to $Z_2$\;
            Let $F_{U_1} = \Algo(\cl{U_1},e_{U_1},\true)$, $F_{V_2} = \Algo(\cl{V_2},e_{V_2},\true)$, $F_{U_2} = \Algo(\cl{U_2},e_{U_2},\false)$,  $F_{V_1} = \Algo(\cl{V_1},e_{V_1},\false)$ (for each of $U_1,V_2,U_2,V_1$ that is nontrivial), and $F_Z = \Algo(Z+u_1'v_2',u_1'v_2',\true)$\;
            \Return{ $F = (F_{U_1} - e_{U_1}) \cup (F_{V_2} - e_{V_2}) \cup F_{U_2} \cup F_{V_1} \cup (F_Z - u_1'v_2') + \{u,v\} + \{uv\} + \{\text{end edges of $U_1$ and $V_2$}\}$}\;
            }
            \uElse{
            Relabel vertices as necessary (possibly exchanging $u$ and $v$) so that $\Del(U_1) + \Del(V_i) + \BDel(U_2) + \BDel(V_{3-i}) \le 0$ for $i \in [2]$\;
            Let $F_Z = \subroutine(Z_1,u_1',v_1',v_2')$\;
            Relabel so that $F_Z \in \ec(Z_1+u_1'v_2',u_1'v_2')$\;
            Let $F_{U_1} = \Algo(\cl{U_1},e_{U_1},\true)$, $F_{V_2} = \Algo(\cl{V_2},e_{V_2},\true)$, $F_{U_2} = \Algo(\cl{U_2},e_{U_2},\false)$, and $F_{V_1} = \Algo(\cl{V_1},e_{V_1},\false)$ (for each of $U_1,V_2,U_2,V_1$ that is nontrivial)\;
            \Return{$F = (F_Z - u_1'v_2') \cup (F_{U_1} - e_{U_1}) \cup (F_{V_2} - e_{V_2}) \cup F_{U_2} \cup F_{V_1} + \{u,v\} + \{uv\} + \{\text{end edges of $U_1$ and $V_2$}\}$}\;
            }
        \caption{\label{algo:ec_1} $\alg(G,e,-1)$}
	\end{algorithm}

	\begin{prop}\label{prop:algo_ec_1}
	Suppose $\Algo$ is correct on all inputs of size less than $n=n(G)$. Then $\alg(G,e,-1)$ is correct and takes $\falg(n-1)+O(n)$ time for all inputs of size less than or equal to $n$.
	\end{prop}

	\begin{proof}
	The proof of correctness follows the same structure of Section \ref{sec:mainproof}.
	First suppose $e$ has a parallel edge. Then $F\in\ec(G,e)$ and $\exc(F)=\exc(F')$, and since $n(G)=n(G')+2$, $n_2(G)=n_2(G')$, and $\Del(G',e')\leq -\frac12$, we have by induction $\exc(F)\leq \frac{n(G')+n_2(G')}{4}+\Del(G',e')+2\leq \frac{n(G)+n_2(G)}{4}+1$; moreover, this takes $\falg(n-2)+O(1)$ time. So we may assume that $G$ is simple.
	The existence of $Z_1$ and $Z_2$ follows from Claim \ref{clm:Z_1Z_2}, and they can be determined from the block structure of $G-\{u,v\}$ in linear time.  As $\Del(U_i) = -\BDel(U_i)$ and $\Del(V_i) = -\BDel(V_i)$ for $i \in [2]$, the relabeling on lines 5 and 10 are always possible (on line 10, after exchanging $u$ and $v$ if necessary, we may assume that $|\Del(U_1)-\Del(U_2)|\geq|\Del(V_1)-\Del(V_2)|$). In what follows, for each trivial chain $C\in\{U_1,V_2\}$, we read $\exc(F_C)-2$ as $0$ (as $F_C$ is empty and the cycle of $F$ containing $e$ uses the edge of $C$); recall that $n(\cl C)=n_2(\cl C)=0$ for every trivial chain $C$. Furthermore, regardless of whether $Z_1 \neq Z_2$ or $Z_1 = Z_2$, we have
	\begin{itemize}
	    \item $\exc(F) -2 = (\exc(F_{U_1}) - 2) + (\exc(F_{V_2}) - 2) + \exc(F_{U_2}) + \exc(F_{V_1}) + (\exc(F_Z) - 2)$,
	    \item $n(G) = n(\cl{U_1}) + n(\cl{V_2}) + n(\cl{U_2}) + n(\cl{V_1}) + n(Z+u_1'v_2') + 2$, and
	    \item $n_2(G) = n_2(\cl{U_1}) + n_2(\cl{V_2}) + n_2(\cl{U_2}) + n_2(\cl{V_1}) + n_2(Z+u_1'v_2') - 2$.
	\end{itemize}

	By induction, we have $\exc(F_{U_1}) - 2 \le \frac{n(\cl{U_1}) + n_2(\cl{U_1})}{4} + \Del(U_1)$, $\exc(F_{V_2}) - 2 \le \frac{n(\cl{V_2}) + n_2(\cl{V_2})}{4} + \Del(V_2)$, $\exc(F_{U_2}) \le \frac{n(\cl{U_2}) + n_2(\cl{U_2})}{4} + \BDel(U_2)$, and $\exc(F_{V_1}) \le \frac{n(\cl{V_1}) + n_2(\cl{V_1})}{4} + \BDel(V_1)$. We argue now that in both cases we have
	\begin{align}\label{eqn:Z}
	    \exc(F_{Z}) - 2 \le \frac{n({Z+u_1'v_2'}) + n_2({Z+u_1'v_2'})}{4} - 1.
	\end{align}
	If $Z_1 = Z_2$, this follows from Proposition \ref{prop:algo_Z_1=Z_2}. If $Z_1 \neq Z_2$, then $(Z+u_1'v_2',u_1'v_2')$ is the closure of a subcubic chain with at least two blocks, namely $Z_1$ and $Z_2$. By induction on its chain-blocks, we have
	\begin{align*}
	    \exc(F_Z) - 2 \le \frac{n(Z+u_1'v_2') + n_2(Z+u_1'v_2')}{4} + \Del(Z+u_1'v_2',u_1'v_2')  \le \frac{n(Z+u_1'v_2') + n_2(Z+u_1'v_2')}{4} -1
	\end{align*}
	and (\ref{eqn:Z}) holds in both cases. Thus,
    \begin{align*}
        \exc(F) - 2 &= (\exc(F_{U_1}) - 2) + (\exc(F_{V_2}) - 2) + \exc(F_{U_2}) + \exc(F_{V_1}) + (\exc(F_Z) - 2)\\
        &\le \frac{n(G) + n_2(G)}{4} + \Del(U_1)+\Del(V_2) + \BDel(U_2)+\BDel(V_1) - 1\\
        &\le \frac{n(G) + n_2(G)}{4} -1.
    \end{align*}

	For the time complexity, note that we only call $\Algo$ and $\subroutine$ on inputs whose sizes sum to less than $n$. As the remaining lines require $O(n)$ time by Proposition \ref{prop:scan}, we have that the entire algorithm requires $\falg(n-1)+O(n)$ time.
	\end{proof}

	We now present the final case for $\alg$.

		\begin{algorithm}[H]
	    \small
        \SetKwInOut{Input}{Input}
        \SetKwInOut{Output}{Output}

        \Input{A subcubic graph $G$ and $e=uv \in E(G)$ with $G - e$ is simple and 2-connected, and $\Del(G,e)=-\frac{3}{2}$ (i.e. $(G_u,f_u)$ is a rooted $\theta$-chain)}
        \Output{$F\in \ec(G,e)$ with $\exc(F) \le \frac{n(G) + n_2(G)}{4} + \frac{1}{2}$}
        Let $C_1$ and $C_2$ denote the chains of $(G_u,f_u)$ with common endpoints $f_u=\{u_1,u_2\}$ and $v \in V(C_1)$\;
        Let $x_i \in V(C_2)$ be the neighbor of $u_i$ for $i \in [2]$\;
        Write $C_1 = u_1e_0B_1 \ldots e_{k-1}B_ke_ku_2$\;
        Let $\ell\in[k]$ be the unique index such that $v \in V(B_\ell)$\;
        Let $v'$ denote the endpoint of $e_{\ell-1}$ in $B_\ell$, and let $v''$ denote the endpoint of $e_\ell$ in $B_\ell$\;
        Let $D_1$ and $D_2$ denote the chains of $C_1$ with end points $\{u_1,v'\}$ and $\{v'',u_2\}$ respectively\;

        For $i \in [2]$, let $(\Del(D_i),\BDel(D_i)) = \scan(\cl{D_i},e_{D_i})$ if $D_i$ is nontrivial, and $(\Del(D_i),\BDel(D_i))=(0,0)$ otherwise\;
        Relabel if necessary so that $\Del(D_1) + \BDel(D_2) \le 0$\;
        Let $F_2 = \Algo(\cl{C_2},e_{C_2},\true)$, $F_{D,1} = \Algo(\cl{D_1},e_{D_1},\true)$ if $D_1$ is nontrivial and $F_{D,1}=\emptyset$ otherwise, $F_{D,2} = \Algo(\cl{D_2},e_{D_2},\false)$ if $D_2$ is nontrivial and $F_{D,2}=\emptyset$ otherwise, and
        $F_{\ell} = \Algo(B_{\ell} + v'v,v'v,\true)$\;
        \Return{$F = (F_2 - e_{C_2}) \cup (F_{D,1} - e_{D_1}) \cup F_{D,2} \cup (F_{\ell} - v'v) + \{u,u_1,u_2\} + \{e_0,e_{\ell - 1},u_1x_1,uv,uu_2,u_2x_2\}$}\;

        \caption{\label{algo:ec_3/2} $\alg(G,e,-\frac{3}{2})$}
	\end{algorithm}

	\begin{prop}\label{prop:algo_ec_3/2}
	Suppose $\Algo$ is correct for all inputs of size less than $n=n(G)$. Then $\alg(G,e,-\frac32)$ is correct and takes $\falg(n-1)+O(n)$ time for all inputs of size less than or equal to $n$.
	\end{prop}

	\begin{proof}
	We first analyze the correctness of the returned even  cover $F$. If $D_1$ is trivial, we read $\exc(F_{D,1})-2$ as $0$ below (as $F_{D,1}$ is empty and the cycle of $F$ containing $e$ uses the edge of $D_1$); recall that $n(\cl{D_i})=n_2(\cl{D_i})=0$ if $D_i$ is trivial. By induction, we have that $\exc(F_2) \le \frac{n(\cl{C_2}) + n_2(\cl{C_2})}{4} + \Del(\cl{C_2},e_{C_2}) + 2$, $\exc(F_{D,1}) \le \frac{n(\cl{D_1}) + n_2(\cl{D_1})}{4} + \Del(D_1)  + 2$, $\exc(F_{D,2}) \le \frac{n(\cl{D_2}) + n_2(\cl{D_2})}{4} + \BDel(D_2)$, and $\exc(F_{\ell}) \le \frac{n(B_{\ell} + v'v) + n_2(B_{\ell} + v'v)}{4} + \frac{3}{2}$. As $\exc(F)-2 = (\exc(F_2)-2) + (\exc(F_{D,1}) - 2) + \exc(F_{D,2}) + (\exc(F_{\ell}) - 2)$, $n(G) = n(\cl{C_2}) + n(\cl{D_1}) + n(\cl{D_2}) + n(B_{\ell} + v'v) + 3$, and $n_2(G) = n_2(\cl{C_2}) + n_2(\cl{D_1}) + n_2(\cl{D_2}) + n_2(B_{\ell} + v'v) - 1$, we have

	\begin{align*}
	    \exc(F)-2 &= (\exc(F_2)-2) + (\exc(F_{D,1}) - 2) + \exc(F_{D,2}) + (\exc(F_{\ell}) - 2)\\
	    &\le \frac{n(G) + n_2(G)}{4} -\frac12 + \Del(\cl{C_2},e_{C_2}) + \Del(D_1) + \BDel(D_2) + \Del(B_\ell+v'v,v'v) \\
	    &\leq \frac{n(G) + n_2(G)}{4} -\frac32,
	\end{align*}
	since $\Del(\cl{C_2},e_{C_2}), \Del(B_\ell+v'v,v'v) \le -\frac{1}{2}$ and $\Del(D_1) + \BDel(D_2) \le 0$.
	Thus $\exc(F)$ satisfies our desired bound.

	For the time analysis, as we only call $\Algo$ on inputs whose sizes sum to less than $n$, line 9 takes at most $\falg(n-1)$ time. Furthermore, by Proposition \ref{prop:scan}, the remaining lines require $O(n)$ time. Thus, $\alg(G,e,-\frac{3}{2})$ takes $\falg(n-1)+O(n)$ time.
	\end{proof}

	To summarize, we have the following.

	\begin{cor}
	$\Algo$ is correct and takes $O(n^2)$ time.
	\end{cor}
	\begin{proof}
	    We show inductively that we can take $\falg(n)=O(n^2)$.
	    First note that lines 1-5 take constant time.
	    Line 6 takes linear time to check, and if executed, lines 7-12 take $O(n) + \sum_{i=1}^k \falg(n(\cl{B_i})) \leq O(n)+\sum_{i=1}^k O(n(\cl{B_i})^2) = O(n^2)$.

	    Line 13 takes linear time by Proposition \ref{prop:scan}, and in lines 14-17, we execute exactly one of $\alg(G,e,\Del)$ and $\balg(G,e)$, which takes $\falg(n-1)+O(n)$ time by Propositions \ref{prop:algo_bec_3/2}, \ref{prop:ec_1/2}, \ref{prop:algo_ec_1}, and \ref{prop:algo_ec_3/2}.
	    It follows that we can take $\falg(n) = O(n^2)$.
	\end{proof}

	\begin{cor}\label{cor:algo}
	    Given a simple 2-connected subcubic graph $G$, we can find an even cover $F$ of $G$ with $\exc(F)\leq \frac{n(G)+n_2(G)}{4}+1$ in quadratic time.
	\end{cor}
	\begin{proof}
	    Pick an arbitrary edge $e\in E(G)$. Run $\Algo(G,e,\true)$ and $\Algo(G,e,\false)$.
	    One of the returned even covers will have excess at most $\frac{n(G)+n_2(G)}{4}+1$.
	\end{proof}

	Let us now complete the proof of Theorem \ref{thm:tspwalk}, restated here for the reader's convenience.
	\mainthm*
	\begin{proof}
	By Corollary \ref{cor:algo}, we can find an even cover $F$ of $G$ with $\exc(F) \leq \frac{n(G)+n_2(G)}{4}+1$ in quadratic time.
	Then by Proposition \ref{prop:evencoverwalk}, we can convert $F$ to a TSP walk of length $\exc(F)-2+n(G) \leq \frac{5n(G)+n_2(G)}{4}-1$ in linear time.
	\end{proof}
	If the input graph $G$ is cubic (i.e.~$n_2(G)=0$), then Theorem \ref{thm:tspwalk} finds a TSP walk of length at most $\frac{5n(G)}{4}-1$ in quadratic time.
	Since every TSP walk trivially has length at least $n(G)$, this gives a $\frac{5}{4}$-approximation algorithm for TSP walks in 2-connected simple cubic graphs.
	For 2-connected simple subcubic graphs, Theorem \ref{thm:tspwalk} finds a TSP walk of length at most $\frac32 n(G)$ which trivially yields a $\frac32$-approximation algorithm.
	The bound gets better for subcubic graphs with fewer vertices of degree 2; for example, if $n_2(G)\leq \frac13 n(G)$, then Theorem \ref{thm:tspwalk} yields a TSP walk of length at most $\frac43 n(G)$.
	We suspect that refining the ideas developed in this paper could lead to another $\frac43$-approximation algorithm for subcubic graphs, matching the current best ratio by M{\"o}mke and Svensson \cite{MS11}, and possibly beyond.

	\bigskip

    {\bf Acknowledgements}  We would like to thank Swati Gupta for suggesting problems related to the Traveling Salesperson Problem during Michael Wigal's PhD oral examination.

    Michael Wigal was supported by an NSF Graduate Research Fellowship under Grant No. DGE-1650044.
    Youngho Yoo was partially supported by the Natural Sciences and Engineering Research Council of Canada (NSERC), PGSD2-532637-2019.
    Xingxing Yu was partially supported by NSF Grant DMS 1954134.

	AI was used to find and correct the errors in the previous version of this paper.

	\bibliography{refs}

@article{AGG11,
  title={A 4/3-approximation for TSP on cubic 3-edge-connected graphs},
  author={Aggarwal, Nishita and Garg, Naveen and Gupta, Swati},
  journal={arXiv preprint arXiv:1101.5586},
  year={2011}
}

@article{BSSS14,
  title={The traveling salesman problem on cubic and subcubic graphs},
  author={Boyd, Sylvia and Sitters, Ren{\'e} and van der Ster, Suzanne and Stougie, Leen},
  journal={Mathematical Programming},
  volume={144},
  number={1},
  pages={227--245},
  year={2014},
  publisher={Springer}
}

@article{CL18,
  title={Cubic TSP-a 1.3-approximation},
  author={Candr{\'a}kov{\'a}, Barbora and Lukot'ka, Robert},
  journal={arXiv preprint arXiv:1506.06369},
  year={2015}
}

@techreport{C76,
  title={Worst-case analysis of a new heuristic for the travelling salesman problem},
  author={Christofides, Nicos},
  year={1976},
  institution={Carnegie-Mellon Univ Pittsburgh Pa Management Sciences Research Group}
}

@article{CLS15,
  title={TSP tours in cubic graphs: beyond 4/3},
  author={Correa, Jos{\'e} and Larr{\'e}, Omar and Soto, Jos{\'e} A},
  journal={SIAM Journal on Discrete Mathematics},
  volume={29},
  number={2},
  pages={915--939},
  year={2015},
  publisher={SIAM}
}

@inproceedings{DKM17,
  title={Graphic TSP in Cubic Graphs},
  author={Dvo\v{r}\'ak, Zden{\^e}k and Kr{\'a}l, Daniel and Mohar, Bojan},
  booktitle={34th Symposium on Theoretical Aspects of Computer Science (STACS 2017)},
  year={2017},
  organization={Schloss Dagstuhl-Leibniz-Zentrum fuer Informatik}
}

@article{GLS04,
  title={An improved upper bound for the TSP in cubic 3-edge-connected graphs},
  author={Gamarnik, David and Lewenstein, Moshe and Sviridenko, Maxim},
  journal={Operations Research Letters},
  volume={33},
  number={5},
  pages={467--474},
  year={2005},
  publisher={Elsevier}
}

@inproceedings{GSS11,
  title={A randomized rounding approach to the traveling salesman problem},
  author={Gharan, Shayan Oveis and Saberi, Amin and Singh, Mohit},
  booktitle={2011 IEEE 52nd Annual Symposium on Foundations of Computer Science},
  pages={550--559},
  year={2011},
  organization={IEEE}
}

@article{G95,
  title={Worst-case comparison of valid inequalities for the TSP},
  author={Goemans, Michel X.},
  journal={Mathematical Programming},
  volume={69},
  number={1},
  pages={335--349},
  year={1995},
  publisher={Springer}
}

@inproceedings{Ka72,
  title={Reducibility Among Combinatorial Problems},
  author={Karp, Richard},
  booktitle={in R. E. Miller and J. W. Thatcher (editors). Complexity of Computer Computations}, 
  pages={85–-103},
  year={1972},
  organization={New York: Plenum}
}

@article{KS15,
  title={Approximation hardness of graphic TSP on cubic graphs},
  author={Karpinski, Marek and Schmied, Richard},
  journal={RAIRO-Operations Research},
  volume={49},
  number={4},
  pages={651--668},
  year={2015},
  publisher={EDP Sciences}
}

@inproceedings{MS11,
  title={Approximating graphic TSP by matchings},
  author={M{\"o}mke, Tobias and Svensson, Ola},
  booktitle={2011 IEEE 52nd Annual Symposium on Foundations of Computer Science},
  pages={560--569},
  year={2011},
  organization={IEEE}
}

@article{M14,
  title={$\frac{13}{9}$-Approximation for Graphic TSP},
  author={Mucha, Marcin},
  journal={Theory of computing systems},
  volume={55},
  number={4},
  pages={640--657},
  year={2014},
  publisher={Springer}
}

@article{SV14,
  title={Shorter tours by nicer ears: 7/5-approximation for the graph-TSP, 3/2 for the path version, and 4/3 for two-edge-connected subgraphs},
  author={Seb{\H{o}}, Andr{\'a}s and Vygen, Jens},
  journal={Combinatorica},
  volume={34},
  number={5},
  pages={597--629},
  year={2014},
  publisher={Springer}
}

@inproceedings{Z16,
  title={Improved approximations for cubic bipartite and cubic TSP},
  author={van Zuylen, Anke},
  booktitle={International Conference on Integer Programming and Combinatorial Optimization},
  pages={250--261},
  year={2016},
  organization={Springer}
}

@book{cormen2009introduction,
  title={Introduction to algorithms},
  author={Cormen, Thomas H and Leiserson, Charles E and Rivest, Ronald L and Stein, Clifford},
  year={2009},
  publisher={MIT press}
}

@inproceedings{karlin2021slightly,
  title={A (slightly) improved approximation algorithm for metric TSP},
  author={Karlin, Anna R and Klein, Nathan and Gharan, Shayan Oveis},
  booktitle={Proceedings of the 53rd Annual ACM SIGACT Symposium on Theory of Computing},
  pages={32--45},
  year={2021}
}

@article{karpinski2015new,
  title={New inapproximability bounds for TSP},
  author={Karpinski, Marek and Lampis, Michael and Schmied, Richard},
  journal={Journal of Computer and System Sciences},
  volume={81},
  number={8},
  pages={1665--1677},
  year={2015},
  publisher={Elsevier}
}

@article{serdyukov1978nekotorykh,
  title={O nekotorykh ekstremal’nykh obkhodakh v grafakh},
  author={Serdyukov, A. I.},
  journal={Upravlyayemyye sistemy},
  volume={17},
  pages={76--79},
  year={1978}
}

@article{van2020historical,
  title={A historical note on the 3/2-approximation algorithm for the metric traveling salesman problem},
  author={van Bevern, Ren{\'e} and Slugina, Viktoriia A},
  journal={Historia Mathematica},
  volume={53},
  pages={118--127},
  year={2020},
  publisher={Elsevier}
}

@article{benoit2008finding,
  title={Finding the exact integrality gap for small traveling salesman problems},
  author={Benoit, Genevieve and Boyd, Sylvia},
  journal={Mathematics of Operations Research},
  volume={33},
  number={4},
  pages={921--931},
  year={2008},
  publisher={INFORMS}
}
	\bibliographystyle{siam}

\end{document}